\documentclass{article}

\title{\textbf{Deformation of the heat kernel and the Brownian motion from the perspective of the \mbox{Ben Sa\"id}--Kobayashi--\O rsted $\bf{(k,a)}$-generalized Laguerre semigroup theory} 
\\ \vspace{7pt} {\large \textit{Dedicated to Toshiyuki Kobayashi with admiration to his outstanding works which unlock the future of the fields.} }}
\author{By \vspace{10pt} \\  Aoyama Temma\thanks{Graduate School of Mathematical Sciences, the University of Tokyo, 3-8-1 Komaba,
Meguro, 153-8914 Tokyo, Japan. Email address: \href{mailto:horseof3@g.ecc.u-tokyo.ac.jp}{\nolinkurl{horseof3@g.ecc.u-tokyo.ac.jp}}.}}
\date{}

\usepackage{amsmath,amssymb,amsfonts,amsthm}
\usepackage{mathrsfs}
\usepackage{bm}

\usepackage{framed,color}
\usepackage[dvipsnames]{xcolor}
\usepackage{mathtools}
\mathtoolsset{showonlyrefs,showmanualtags}
\usepackage{thmtools}

\usepackage{tikz}
\usetikzlibrary{intersections,calc,arrows.meta}
\usetikzlibrary{quotes,angles}

\usepackage{wrapfig}

\newtheoremstyle{mystyle}
    {5pt}
    {7pt}
    {}
    {0pt}
    {\bf\vspace{5pt}}
    {}
    {\newline}
    {\thmname{#1}\thmnumber{ #2}\thmnote{ (#3)}}

\newenvironment{myleftbar}[1]{
	\def\FrameCommand{{\color{#1}\vrule width 2pt }\hspace{10pt}}%
	\MakeFramed {\advance\hsize-\width \FrameRestore}}%
{\endMakeFramed\par}

\newcommand{\newtheorembar}[2][black]{
	\newenvironment{#2}{
		\begin{myleftbar}{#1}
		\begin{vir#2}}
		{\end{vir#2}\end{myleftbar}}
		\theoremstyle{mystyle}
	\newtheorem{vir#2}}

\newenvironment{myleftbar2}[1]{
	\def\FrameCommand{{\color{#1}\vrule width 1pt }\hspace{10pt}}%
	\MakeFramed {\advance\hsize-\width \FrameRestore}}%
{\endMakeFramed\par}

\newcommand{\pnewtheorembar}[2][black]{
	\newenvironment{#2}{
		\begin{myleftbar2}{#1}
		\begin{vir#2}}
		{\end{vir#2}\end{myleftbar2}}
		\theoremstyle{mystyle}
	\newtheorem{vir#2}}

\newtheorembar[cyan]{ex}{例}

	\newtheorembar[black]{theorem}[sentense]{Theorem}
	\newtheorembar[black]{defprop}[sentense]{Definition-Proposition}
	\newtheorembar{lemma}[sentense]{Lemma}
	\newtheorembar{definition}[sentense]{Definition}
	\newtheorembar{corollary}[sentense]{Corollary}
	\newtheorembar{example}[sentense]{Example}
	\newtheorembar{fact}[sentense]{Fact}
	\newtheorembar{proposition}[sentense]{Proposition}
	\pnewtheorembar{ptheorem}[sentense]{Theorem}
	\pnewtheorembar[black]{plemma}[sentense]{Lemma}
	\pnewtheorembar{pdefinition}[sentense]{Definition}
	\pnewtheorembar{pcorollary}[sentense]{Corollary}
	\pnewtheorembar{pexample}[sentense]{Example}
	\pnewtheorembar{pfact}[sentense]{Fact}
	\pnewtheorembar{pdef}[sentense]{Definition}
	\pnewtheorembar{pthm}[sentense]{Theorem}

\usepackage{hyperref}
\hypersetup{%
 setpagesize=false,%
 bookmarks=true,%
 bookmarksdepth=tocdepth,%
 bookmarksnumbered=true,%
 colorlinks=true,%
 linkcolor=blue,%
 citecolor=teal,%
 urlcolor=PineGreen,%
  pdftitle={},%
 pdfsubject={},%
 pdfauthor={},%
 pdfkeywords={},%
 }

\begin{document}
\maketitle

\begin{abstract}
\noindent
We deform the heat kernel and the Brownian motion on $\mathbb{R}^{N}$ from the perspective of ``$(k,a)$-generalized Fourier analysis'' with $k=0$.
This is a new type of harmonic analysis proposed by \mbox{S.Ben Sa\"id}--T.Kobayashi--B.\O rsted from the representation theoretic viewpoint. 
In this paper, we construct the $a$-deformed heat kernel and $a$-deformed Brownian motion, and explore their some basic properties. 
We also prove that the $(k,a)$-generalized Fourier integral kernels are polynomial growth when $k=0$, 
for a justification of some discussions.

\end{abstract}

\setcounter{tocdepth}{1}
\tableofcontents

\vspace{7pt}
\section{Introduction}\label{section1}
\subsection{Context of the study}

  \vspace{0pt}
  In this subsection, we review in brief the context of the study with some historical notes. 
  Readers who are already familiar with this background may skip this subsection.

  \vspace{10pt}
{\normalsize\textbf{The oscillator semigroup}}\par
  Representation theory sometimes offers a glimpse of the hidden mechanism of mathematical phenomena. 
This may encourage us to reinterpret familiar objects more deeply or even help to encounter new mathematics.

It is well known that the classical harmonic analysis is one of the examples having such a mechanism --- called the oscillator semigroup.
 
 \vspace{5pt}
 The first key point of the theory of oscillator semigroup is the set of differential operators on $\mathbb{R}^{N}$,
 $$ \mathrm{span}_{\mathbb{R}}\left\langle \sqrt{-1}x_{i}x_{j} ,\hspace{3pt} \sqrt{-1}\frac{\partial^2}{\partial x_{i} \partial x_{j}},\hspace{3pt}x_{i}\frac{\partial}{\partial x_{j}}+\frac{\partial}{\partial x_{j}}x_{i}\hspace{1pt} \left.\right|\, i,j = 1,\dots N \right\rangle $$
 is closed under Lie bracket and isomorphic to $\mathfrak{sp}(N,\mathbb{R})$.  
 This defines an infinitesimal unitary action of $\mathfrak{sp}(N,\mathbb{R})$ on $L^{2}(\mathbb{R}^{N},dx)$.  
 
 $\mathrm{O}(N)$-invariant part of this Lie algebraic action is especially important. 
 This consists of $\mathbb{R}$-span of
 $$ \frac{i}{2}|x|^{2}, \hspace{10pt}\frac{i}{2}\Delta ,\hspace{10pt} \sum_{i=1}^{N}x_{i}\frac{\partial}{\partial x_{i}}+\frac{N}{2},$$
 and, in fact, these three differential operators form $\mathfrak{sl}_{2}$-triple. 
 
  These Lie algebraic actions of  $\mathfrak{sp}(N,\mathbb{R})$ was firstly constructed explicitly by Van Hove \cite{MR0057260}, according to \cite{MR0983366}.

  \vspace{10pt}
  From this perspective, we may expect to be able to apply Lie theoretic technique to the analysis related to these differential operators. 
  This expectation comes true. 
  
  That is, the action of $\mathfrak{sp}(N,\mathbb{R})$ are able to be lifted to the action of the metaplectic group $\mathrm{Mp}(N,\mathbb{R})$. 
  This lifted up action is called the oscillator representation (or Metaplectic representation, Segal--Shale--Weil representation, Harmonic representation). 
  Moreover, this action is extended to the holomorphic action which includes $\mathrm{Mp}(N,\mathbb{R})$ action as a boundary. 
  This extended action is called the oscillator semigroups. 
  
  The oscillator semigroup includes important operators; 
  \begin{center}
  \begin{tabular}{c c l}
$e^{-\frac{z}{2}||x||^{2}} $& for $\mathrm{Re}(z) >0$ & : multiplying the Gaussian function \\
$e^{\frac{z}{2}\Delta} $ & for $\mathrm{Re}(z) >0$ &: heat flow operator \\
$e^{\frac{z}{2}(\Delta-||x||^{2})} $ & for $\mathrm{Re}(z) >0$ &:  Hermite semigroup \\
$i^{N/2}e^{\frac{\pi i}{4}(\Delta-||x||^{2})} $ &  &: Fourier transform. \\
\end{tabular}
\end{center}

  As a result, several properties of these operators are translated to the group (and semigroup) theoretic language.

  According to \cite{MR0983366}, the theory of the oscillator representation has its origin in the study of I.\,Segal \cite{MR0128839}, D.Shale \cite{MR0137504}, and A.Weil \cite{MR0165033}. 
  I.Segal and D.Shale were motivated by quantum mechanics and A.Weil had the purpose to apply it to number theory. 
  The extension to holomorphic semigroup were extensively studied by R.Howe \cite{MR0974332}.
  \vspace{20pt}

  \textbf{Geometric Analysis of Minimal representation}\par
  The oscillator representation is also important in unitary representation theory, 
  because it is the minimal representation of C-type. 
  The above realization is called Schr\"{o}dinger model of minimal representation of C-type. 
  It is known that minimal representations are relatively small representation though it is infinite-dimensional. 
  In other words, it is the representation with large symmetry. This principle was suggested by T.Kobayashi, see e.g. \cite{MR2849643, MR4867008} for survey. 
  The minimal representations are expected to play a role of one of the building blocks of unitary representations (unipotent representation). 
  Additionally, minimal representation itself has very attractive properties algebraically, geometrically, analytically, and representation theoretically, so it is also expected to have a role as a junction of various mathematics.

  \vspace{10pt}

Since 1990s, several researchers studied its existence, unitarity, 
and construction with algebraic approaches \cite{MR1108044, MR1631302, MR1278630, MR1267034, MR1327538, MR1103588, MR1484858, MR1387517, MR1253210} and recently, 
the classification has been carried out \cite{MR2123125, MR3968907}. 
This is very fruitful and fascinating topics in mathematics. 

At the same time, geometric and analytic study of minimal representation launched:
T.Kobayashi--B.\O rsted studied a representation theory related to conformal geometry of celestial spheres, 
and from this perspective, found out the geometric analytic realization of the minimal representation of $\mathrm{O}(p,q)$, 
to be called Schr\"{o}dinger model of it in 1991-1992 \cite{MR1649917, MR1982432, MR2020550, MR2020551, MR2020552} 
which were constructed by a different and algebraic approach by B.Kostant \cite{MR1103588} for $(p,q)=(4,4)$ and B.Binegar--R.Zierau \cite{MR1108044} for the general $p,q \geqq 2 $. 
T.Kobayashi--G.Mano conducted a thorough analysis of the Schrödinger model of minimal representation, specifically proving an explicit formula for the unitary inversion operator (generalized Fourier transform) in their work \cite{MR2317306}. 
In addition to that, they found, when the case of $q=2$, the holomorphic semigroup which is similar to the Hermite semigroup are able to be constructed \cite{MR2401813}. 
They coined it as the `Laguerre semigroup' and interpreted the unitary inversion operator (generalized Fourier transform) as its boundary value in the space of operators.
This work was extended by \mbox{S.Ben Sa\"id}--T.Kobayashi--B.\O rsted \cite{MR2566988, MR2956043}, 
where they introduced the deformation of the Laguerre semigroup including the Hermite semigroup and the semigroup related Dunkl operator \cite{MR951883, MR1199124}, 
and discuss the ``$(k,a)$-generalized Fourier transform" from this perspective. 
J.Hilgert--T.Kobayashi--G.Mano--J.M\"ollers discovered the new types of special functions associated with fourth-order differential equations through the study of analysis of minimal representation \cite{MR2860690,MR2837716}. 
In this movement, T.Kobayashi advocated the new area "Geometric Analysis of Minimal representation" \cite{MR2849643, MR4867008}.

\vspace{10pt}
\textbf{$\bf{(k,a)}$-generalized Fourier analysis}\par
The content of this article is based on the study of the ``$(k,a)$-generalized Laugurre semigroup" and the ``$(k,a)$-generalized Fourier transform" \cite{MR2401813, MR2566988, MR2956043}, the case when $k=0$.

\vspace{5pt}
 In \cite{MR2401813}, the minimal representation of $\mathrm{O}(N+1,2)$ was investigated. They found a similar nature to the oscillator representation such as  
 the certain set of differential operators with degree less than 2 which is closed under Lie bracket and isomorphic to $\mathfrak{o}(N+1,2)$. 
 Using this, they considered an infinitesimal unitary representation on $L^{2}(\mathbb{R}^{N},\frac{dx}{|x|})$. 
 The $\mathrm{O}(N)$-invariant part of this action is spanned by
 
  $$ i |x|, \hspace{10pt}i |x|\Delta ,\hspace{10pt} 2\sum_{i=1}^{N}x_{i}\frac{\partial}{\partial x_{i}}+N-1,$$
 and these form $\mathfrak{sl}_{2}$-triple. 
 The Laguerre semigroup and the generalized Fourier transform are defined as follows;
  \begin{center}
  \begin{tabular}{c c l}
$e^{z(|x|\Delta-|x|)} $ & for $\mathrm{Re}(z) >0$ &: Laguerre semigroup \\
$i^{N-1}e^{\frac{\pi i}{2}(|x|\Delta-|x|)} $ &  &: Generalized Fourier transform. \\
\end{tabular}
\end{center}
 
\vspace{3pt}
 In \cite{MR2566988, MR2956043}, the deformation of the above construction is carried out. In these papers, the differential operators
  $$ \frac{i}{a}|x|^{a}, \hspace{10pt}\frac{i}{a}|x|^{2-a}\Delta_{k} ,\hspace{10pt} \frac{2}{a}\sum_{i=1}^{N}x_{i}\frac{\partial}{\partial x_{i}}+\frac{N+2\langle k \rangle +a-2}{a},$$
 are considered (Here, $k$ is the parameter which relates Dunkl theory \cite{MR951883, MR1199124}. In this article, we study the case when $k=0$. In that case, $\Delta_{k}=\Delta$ and $\langle k\rangle =0$.). 
 These form an $\mathfrak{sl}_{2}$-triple. 
 When $(k,a)=(0,2)$, this coincide with the $\mathfrak{sl}_{2}$-triple appear in the oscillator representation, when $(k,a)=(k,2)$ this relates Dunkl theory \cite{MR951883, MR1199124}, and
 when $(k,a)=(0,1)$ this coincide with the $\mathfrak{sl}_{2}$-triple appear in the theory of Kobayashi--Mano \cite{MR2401813}. 
 From this perspective, we may define the $(k,a)$-generalized Laguerre semigroup and the $(k,a)$-generalized Fourier transform as follows;
 
  \begin{center}
  \begin{tabular}{c c l}
$e^{\frac{z}{a}(|x|^{2-a}\Delta_{k}-|x|^{a})} $ & for $\mathrm{Re}(z) >0$ &: $(k,a)$-generalized Laguerre semigroup \\
$i^{\frac{N+\langle k\rangle +a-2}{a}}e^{\frac{\pi i}{2a}(|x|^{2-a}\Delta_{k}-|x|^{a})} $ &  &: $(k,a)$-generalized Fourier transform. \\
\end{tabular}
\end{center}

\vspace{7pt}
In \cite{MR2401813}, T.Kobayashi--G.Mano found out the explicit formula of the integral kernel of the Laguerre semigroup and the generalized Fourier transform by using the Bessel function, 
and in \cite{MR2566988, MR2956043}, S.\mbox{Ben Sa\"id}--T.Kobayashi--B.\O rsted found out the formula of the integral kernel of the $a$-generalized Laguerre semigroup and the $a$-generalized Fourier transform written as infinite sum of special functions. 

Now, with the $(k,a)$-generalized Fourier transform $\mathscr{F}_{k,a}:= i^{\frac{N+2\langle k \rangle +a-2}{a}}e^{\frac{\pi i}{2a}(|x|^{2-a}\Delta_{k}-|x|^{a})}$, 
we are ready to consider the ``$(k,a)$-generalized Fourier analysis''. 
We may say that the geometric analysis of the minimal representations reveal the framework of a new classical harmonic analysis.

\vspace{10pt}
\textbf{The purpose of this article}\par
Now, we are ready to explain the purpose of this article. From the perspective of \mbox{Ben Sa\"id}--Kobayashi--\O rsted theory \cite{MR2956043},
it may be natural to consider $e^{\frac{t}{a} |x|^{2-a} \Delta_{k}}$ as the $(k,a)$-generalized heat flow operator. 
The following questions may rise: 
Are we able to define this? 
Does it have an integral kernel like the ordinary heat kernel? 
Does the $(k,a)$-generalized heat equation have some unique properties? 
Further, if the answers are positive, 
we may expect the possibility to use the $(k,a)$-generalized heat theory 
as a useful tool of the $(k,a)$-generalized harmonic analysis or to consider the generalization of the various applications of the ordinary heat kernel. 

The $(k,a)$-generalized heat theory was studied in the case of $(k,2)$ by M.R\"{o}sler \cite{MR1620515} and in the case of $(k,1)$ by \mbox{S.Ben Sa\"id}--L.Deleaval \cite{MR4129081}. There are also unpublished studies by T.Kobayashi in 2008 for the case of $(0,a)$ with the method of holomorphic semigroup.
In this article, I would like to study the case of $(0,a)$,
extending partially their studies and to attack the above questions and expectations.

\subsection{Summary of this article}

As explained in last of the former subsection, 
our purpose is to explore the basic properties of 
the $a$-generalized heat theory associated with the $a$-deformed heat flow operator $e^{\frac{t}{a}|x|^{2-a}\Delta}$ and to examine the possibilities of applications. 
Here $e^{\frac{t}{a}|x|^{2-a}\Delta}$ rise from \mbox{Ben Sa\"id}--Kobayashi--\O rsted theoretic point of view \cite{MR2956043}. 

Roughly speaking, the contents of this article are the following:

\begin{enumerate}
\item Estimating the growth of the $(0,a)$-generalized Fourier integral kernel. 
 \item Defining $a$-deformed heat kernel by the $(0,a)$-generalized Fourier transform. 
\item Expanding the $a$-deformed heat kernel with special functions. 
\item Defining $e^{\frac{t}{a}|x|^{2-a}\Delta}$ from the $a$-deformed heat kernel. 
\item Maximum principle of the $a$-deformed heat equation. 
\item Construction of $a$-deformed Brownian motion. 
 \end{enumerate}
 
\vspace{5pt}
 I would like to explain each topics in little more detail below. 
 
 In section \ref{section2}, we review the theory of the $(k,a)$-generalized Laguerre semigroup and the $(k,a)$-generalized Fourier transform of the case when $k=0$, based on 
 \cite{MR2956043}.
 
 \vspace{10pt}
 In section \ref{section3}, we estimate the integral kernels of the $(0,a)$-generalized Laguerre semigroup and the $(0,a)$-generalized Fourier transformation. 
 In the study of the $a$-deformed heat theory, we use the methods of the $(0,a)$-generalized Fourier analysis. 
 Thus, the estimates of growth of the integral kernels are indispensable to justify various operations. 
 In the explicit formula of the $(k,a)$-generalized Laguerre semigroup  calculated by \mbox{S.Ben Sa\"id}--T.Kobayashi--B.\O rsted \cite{MR2956043}, the function 
 \begin{equation}
 	\mathscr{I}(b,\nu, w, t) : = \Gamma(b\nu +1)
	\sum_{m=0}^{\infty}\frac{m+\nu}{\nu} \left(\frac{w}{2}\right)^{bm} \tilde{I}_{b(m+ \nu)}(w)C_{m}^{\nu}(t)\label{I}
\end{equation}
mainly used (See Fact \ref{main2}). 
Since this is the infinite sum of special functions, the estimates are not easy. 
In this article, we derive the complex integral formula (Theorem \ref{main3}) of $\mathscr{I}$ and carry out its estimate (Theorem \ref{Iest1}). 
Then it turns out that the integral kernel of $(0,a)$-generalized Fourier transform has polynomial growth at infinity (Corollary \ref{3cor}). 
By this fact, we are ready to start the $(0,a)$-generalized Fourier analysis. As for estimates of $\mathscr{I}(b,\nu, w, t)$ or $B_{a}(x,y)$ (Theorem \ref{Iest1}, Corollary \ref{3cor}), there are also studies of \cite{MR4871318} and \cite{MR4832104}. 
We see the comparison of these results in the remark after Corollary \ref{3cor}. 

\vspace{10pt}
In section \ref{section4_1}, we define the $a$-deformed heat kernel $h_{a}(x,y;t)$. 
This is defined from the $(0,a)$-generalized Fourier transform of the multiplication operator of the $(0,a)$-generalized Gaussian function $e^{-\frac{t}{a}|x|^{a}}$ (Definition-Proposition \ref{heatdef}).

In section \ref{section4_2}, we explore the more concrete formula of the $a$-deformed heat kernel. 
By carrying out the calculation of special functions, 
we obtain the formula which expresses $h_{a}(x,y;t)$ by elementary functions and the function $\mathscr{I}$ (Theorem \ref{4.1thm}). 
From this formula, we come to know that $h_{a}(x,y;t)$ is real valued and is symmetric about $x,y$. 
Additionally, combining the estimate of $\mathscr{I}$ (Theorem \ref{Iest1}, Theorem \ref{Iest2}), we obtain 
the estimate of $h_{a}(x,y;t)$ similar to the ordinary heat kernel (Corollary \ref{4cor}). 
Some integral formulas which are used later are also proved as corollaries (Corollaries \ref{inteq}, \ref{intpoly}, \ref{intx}).

In section \ref{section4_3}, we consider the one-parameter family of integral operators $\{H_{t}\}_{t\geqq 0}$ defined from $h_{a}(x,y;t)$ (Definition-Proposition \ref{defheatflow}). 
We prove that the flow generated by it satisfies the $a$-generalized heat equation (Corollary \ref{heatfloweq}), and has initial value condition (Theorem \ref{basic}). 
By these properties, we may regard these one-parameter family of integral operators as the $a$-deformed heat flow operator ``$e^{\frac{t}{a}|x|^{2-a}\Delta}$''. 
By the estimate obtained in section \ref{section4_2} (Corollary \ref{4cor}), we may apply the $a$-deformed heat flow operator to polynomial growth functions. 
This is the setting beyond the $L^{2}$-analysis, though the $a$-generalized heat theory born from the unitary representation theory.

In section \ref{section4_4}, we prove the maximum principle of the $a$-generalized heat equation. It claims that 
``the maximum value of a heat flow never exceed the maximum value of the initial state'' (Theorem \ref{maxb} and Theorem \ref{maxu}). 
The proof of them are similar to the one for the ordinary heat equation, but because $a$-deformed one has a singularity at $x=0$, we need to be careful. 
As the classical theory, the maximum principle has fruitful corollaries. 
It leads the positivity of $h_{a}(x,y;t)$ (Corollary \ref{positivity}). 
This positivity lead the property of $a$-deformed heat flow operator $\{H_{t}\}_{t\geqq 0}$: this preserves the magnitude relation (Corollary \ref{supinf}), and the polynomial growthness (Corollary \ref{polygr}).
The uniqueness of solution of $a$-deformed heat equation is also proved from the maximum principle (Corollary \ref{uniq}).
The uniqueness theorem lead the composition law of $h_{a}(x,y;t)$ (Corollary \ref{comp}), and the polynomial type $a$-deformed heat flow formulas (Corollary \ref{polytype}). 
I need to remark that, because of singulality at $x=0$, we are forced into necessity to impose the unconfortable condition on the bahavior of the solution of the $a$-deformed heat equation at $x=0$. 
Now, this condition is partially removed, 
but as I do not know a complete answer yet, 
these discussions are omitted in this article 
(See the Remark after uniqueness theorem (Corollary \ref{uniq})).

\vspace{10pt}
In section \ref{section5}, we construct the $a$-deformed Brownian motion (Definition \ref{defbro}, Theorem \ref{Bro}, and Definition-Proposition \ref{standard}).  
As the Gaussian distribution essentially appear in the definition of the original Brownian motion, we consider to exchange it by the $a$-deformed heat kernel.
That is, one of our purpose in this section is to construct the family of probability measures $\{P_{x}\}_{x\in\mathbb{R}^{N}}$ on $W:=\mathrm{C}\left([0,\infty),\mathbb{R}^{N}\right)$ which are characterized by the following properties:

\begin{enumerate}
\item $P_{x}(\{\gamma\in W\left.\right| \gamma(0)=x\})=1$
\item
For $0=t_{0} < t_{1} < \dots <t_{p}$ and $A_{1}, 
				\dots A_{p} \in \mathfrak{B}(\mathbb{R}^{N})$
\begin{align}
				\MoveEqLeft P_{x}\left(\left\{\gamma\in W\hspace{3pt}|\hspace{3pt} \gamma(t_{i})\in A_{i} \hspace{5pt} \mathrm{for}  
			\hspace{5pt} i = 1, \dots, p\right\}\right) & \\
				& =  \int_{A_{1}} \frac{dx_{1}}{|x_{1}|^{2-a}} 
			\dots \int_{A_{p}}\frac{dx_{p}}{|x_{p}|^{2-a}} \prod_{i=1}^{p} h_{a}(x_{i-1},x_{i}, t_{i}-t_{i-1})
			\end{align}

		Here, $x_{0}=x$ and $\mathfrak{B}(\mathbb{R}^{N})$ is the set of Borel measurable sets 
		with respect to the topology of Euclidean space.
\end{enumerate}

We call these measures as the $a$-deformed Wiener measures.
For the construction of the $a$-deformed Wiener measures, 
the key points are the properties of $h_{a}(x,y;t)$ which we prove in Section \ref{section4} (positivity, composition law and total integral) 
and the application of Kolmogorov's extension theorem (Fact \ref{ext}) and Kolmogorov's continuity theorem (Fact \ref{conti}). 

After the costruction, we prove some Markov properties of $a$-deformed Brownian motion (Proposition \ref{markov1}) and the Feynman-Kac type formula (Theorem \ref{fey}).
\vspace{5pt}

Section \ref{App} is appendix. The properties of special functions used in this article are reviewed.

  \vspace{5pt}
  
  \subsection{Acknowledgement}
  The author would like to express his gratitude to his supervisor, Professor Toshiyuki Kobayashi, for his constant support and encouragement. The author also would like to thank the anonymous referee for making helpful comments. 
  \par
  This research was supported partially by JSPS KAKENHI Grant Number JP24KJ0937 and Forefront Physics and Mathematics Program to Drive Transformation (FoPM), a World-leading Innovative Graduate Study (WINGS) Program, The University of Tokyo.
  
  \vspace{10pt}

\section{Background and Notations}\label{section2}
   \subsection{List of Notations}
   \vspace{5pt}

    Throughout this article, some letters and symbols are used for fixed meaning. 
    When the reader is lost, come back here and check the following list.

   \vspace{5pt}
   \begin{tabular}{|c |c| c|}\hline
      Symbol   & Attribute & Meaning  \\ \hline
      $a$ & $\mathbb{R}_{>0}$ &Parameter for deformation          \\
      $N$ & $\mathbb{N}$ & Dimension of the Euclidean space \\  
      $\lambda_{a}$ & $\mathbb{R}$ & $= \frac{N-2}{a}$ \\
      $\lambda_{a,m}$ & $\mathbb{R}$ & $= \frac{N+2m-2}{a}$ \\
      $c_{a}$ & $\mathbb{R}_{>0}$ & $= \left(a^{\lambda_{a}}\Gamma\left(\lambda_{a}+1\right)\mathrm{vol}\left(S^{N-1}\right)\right)^{-1}$ \\
      $M$ & $\mathbb{R}$ & $=\left(a-1\right)\frac{N-2}{2} +a$ \\ 
      $\Lambda_{a}(x,y;t)$ & Distribution kernel & Integral kernel of the $a$-generalized Laguerre kernel\\
      $B_{a}(x,y)$ & Distribution kernel & Integral kernel of the $a$-generalized Fourier transform\\
      $ J_{a}(w)$ & Special function & Bessel function\\
      $ I_{a}(w)$ & Special function & I-Bessel function\\
      $ \tilde{J}_{a}(w)$ & Special function & Normalized Bessel function\\
      $ \tilde{I}_{a}(w)$ & Special function & Normalized I-Bessel function\\
      $ C^{\nu}_{m}(t)$ & Special function &Gegenbauer polynomial \\
      $ \check{C}^{\nu}_{m}(t)$ & Special function & Normalized Gegenbauer polynomial \\
      $ P_{m}(\omega,\mu)$ & Special function & Zonal spherical harmonics \\
      $ \mathscr{H}^{m}\left(\mathbb{R}^{N}\right)$ & Vector space & Space of Harmonic polynomials with degree $m$\\
      $ x=r\omega, y=s\mu$ & Coordinate & (Ordinary) Polar Coordinate\\
      $ R,S $ & Variables & $ R= \frac{1}{a}r^{a},\hspace{5pt} S = \frac{1}{a}s^{a} $ \\
      $ \mathbb{C}_{+}$ & Semigroup & Complex numbers with non-negative real part\\
      \hline
   \end{tabular}

%

   \vspace{15pt}
   \subsection{Background}
   We would like to review some backgrounds in brief. 
   This part is mainly based on \cite{MR2956043}, so if the reader is interested in the subject or the  original information, see that paper. 
   We work in the case $k=0$ for the $(k,a)$-deformation in \cite{MR2956043}, where $k$ stands for the parameter of the Dunkl theory \cite{MR951883, MR1199124}. 
   \vspace{10pt}

   Firstly, the following Lie algebraic structure of the differential operators is important.

   \begin{fact}[{\cite[Theorem 3.2]{MR2956043}}]
      Consider the differential operators on $\mathbb{R}^{N}$,
      $$\mathbb{E}_{a}^{+}:=\frac{i}{a}|x|^{a}, \hspace{5pt} \mathbb{E}_{a}^{-}:=\frac{i}{a}|x|^{2-a}\Delta , \hspace{5pt} \mathbb{H}_{a} := \frac{2}{a}\sum_{i=1}^{N}x_{i}\partial_{i}+ \frac{N+a-2}{a}.$$
      Then, these form an $\mathfrak{sl}_{2}$-triple, that is
      $$ [\mathbb{H}_{a}, \mathbb{E}_{a}^{+}] = 2 \mathbb{E}_{a}^{+}, \hspace{7pt}[\mathbb{H}_{a}, \mathbb{E}_{a}^{-}] = 2 \mathbb{E}_{a}^{-}, \hspace{7pt} [\mathbb{E}_{a}^{+},\mathbb{E}_{a}^{-}] = \mathbb{H}_{a} .$$
   \end{fact}

   We write this correspondence by
   $$\omega_{a}: \mathfrak{sl}_{2}(\mathbb{R}) \longrightarrow  \mathrm{End}_{\mathbb{C}}\left(C^{\infty}\left(\mathbb{R}^{N}\backslash\{0\}\right)\right)$$

   \vspace{7pt}
   The second important point is the differential operator $\frac{1}{a}\left(|x|^{2-a}\Delta - |x|^{a}\right)$ corresponding 
   $i\begin{pmatrix}
         0 & 1 \\
         -1 & 0
      \end{pmatrix}$
   is the essentially self-adjoint operator on $L^{2}\left(\mathbb{R}^{N},\frac{dx}{|x|^{2-a}}\right)$. 
   Moreover the spectrum is discrete by the following theorem.

   \vspace{10pt}
   Let $\mathscr{H}^{m}(\mathbb{R}^{N})$ be the spaces of harmonic polynomials of degree $m$, $\left\{h_{j}^{(m)}\right\}_{j\in J_{m}}$ be an orthonormal basis, and $L^{(\lambda)}_{l}(x)$ be Laguerre polynomials (For definition, see Appendix \ref{App}). 

   \begin{fact}[{\cite[Corollary 3.17, Proposition 3.19]{MR2956043}}]\label{basisL2}
      Let us consider the functions
      $$\Phi^{(a)}_{l,m,j}(x):= \left(  \frac{ 2^{\lambda_{a,m}} \Gamma (l+1) } {a^{\lambda_{a,m}} \Gamma(\lambda_{a,m} +1) } \right)^{1/2} L^{\left(\lambda_{a,m}\right)}_{l}\left(\frac{2}{a}|x|^{a}\right)e^{-\frac{1}{a}|x|^{a}}h^{(m)}_{j}(x)$$
      \hspace{17pt}for $l, m\in\mathbb{N},\hspace{3pt} j\in J_{m}$, $\lambda_{a,m} := \frac{N+2m-2}{a}$. 

      \vspace{12pt}
      Then,

      \begin{enumerate}
         \item Each $\Phi^{(a)}_{l,m,j}(x)$ are eigenfunctions of $\frac{1}{a}\left(|x|^{2-a}\Delta - |x|^{a}\right)$ :

            $$ \frac{1}{a}\left(|x|^{2-a}\Delta - |x|^{a}\right) \Phi^{(a)}_{l,m,j}(x) = -\left(2l + \lambda_{a,m}+ 1 \right)\Phi^{(a)}_{l,m,j}(x)$$

         \item When $N+a-2 >0$, 

            $\Phi^{(a)}_{l,m,j}(x)$ form an orthonormal basis of $L^{2}\left(\mathbb{R}^{N},\frac{dx}{|x|^{2-a}}\right)$.
         
         \end{enumerate}

   \end{fact}

   In proof, by using the polar coordinate, the following decomposition considered:
   $$\sideset{}{^\oplus}\sum_{m\in\mathbb{N}}\mathscr{H}^{m}\left(\mathbb{R}^{N}\right)|_{S^{N-1}}\otimes L^{2}\left(\mathbb{R}_{+},r^{N+a-2}\frac{dr}{r}\right)\stackrel{\cong}{\longrightarrow}L^{2}\left(\mathbb{R}^{N},\frac{dx}{|x|^{2-a}}\right).$$
   The harmonic polynomials in $\Phi^{(a)}_{l,m,j}(x)$ appear from the spherical components, and the Laguerre polynomials and the exponential in $\Phi^{(a)}_{l,m,j}(x)$ appear from the radial components.

   \vspace{10pt}
   By using the orthonormal basis in Fact \ref{basisL2}, we are able to consider the representation of $\mathfrak{sl}_{2}\left(\mathbb{R}\right)$ on $L^{2}\left(\mathbb{R}^{N},\frac{dx}{|x|^{2-a}}\right)$.
   Let us consider the subspace $W_{a} := \mathrm{span}_{\mathbb{C}}\left\{\Phi^{(a)}_{l,m,j}\right\}$. By Fact \ref{basisL2}, $W_{a}$ is a dense subspace in $L^{2}\left(\mathbb{R}^{N},\frac{dx}{|x|^{2-a}}\right)$. 

   \begin{fact}[{{\cite[Proposition 3.19]{MR2956043}}}]
      $W_{a}$ is stable under $\mathfrak{sl}_{2}\left(\mathbb{R}\right)$-action via $\omega_{a}$.

      $\left(\omega_{a}, W_{a} \right) $ is an infinitesimal unitary representation of $\mathfrak{sl}_{2}\left(\mathbb{R}\right)$.
   \end{fact}

   \vspace{7pt}
   By observing the irreducible decomposition of $\left(\omega_{a},W_{a}\right)$ and using the theory of discretely decomposable representations \cite{MR1637667}, 
   we are able to lift up this representation of Lie algebra to the representation of Lie group.

   \begin{fact}[{\cite[Theorem 3.30]{MR2956043}}]\label{grp}
      
      Suppose $a>\max\left\{0, 2-N\right\}$. Then, there exists an unique unitary representation, to be denoted by $\Omega_{a}$ of $\widetilde{SL}(2,\mathbb{R})$ on $L^{2}\left(\mathbb{R}^{N},\frac{dx}{|x|^{2-a}}\right)$. Such that, 
      $$ \omega_{a}(X) = \left.\frac{d}{dt}\right|_{t=0}\Omega_{a}\left(\mathrm{Exp}(tX)\right)$$
      on the dense subspace $W_{a}\left(\mathbb{R}^{N}\right)$ of  $L^{2}\left(\mathbb{R}^{N},\frac{dx}{|x|^{2-a}}\right)$.
   
   \end{fact}

   \vspace{7pt}
   From the viewpoint of holomorphic semigroup theory, we are able to extend this representation more. 
   Let us consider the closed cone 
   $$W:=\left\{\begin{pmatrix}
      a & b \\
      c & d
   \end{pmatrix}; \hspace{3pt} a^2+bc \leqq 0, \hspace{3pt} b \geqq c\right\}$$
   and $\Gamma\left(W\right) : = \mathrm{SL}\left(2,\mathbb{R}\right)\mathrm{Exp}_{\mathbb{C}}\left(iW\right) \left(= \mathrm{SL}\left(2,\mathbb{R}\right)\mathrm{Exp}_{\mathbb{C}}\left(i\mathbb{R}_{\geqq 0}\begin{pmatrix}
      0 & 1 \\
      -1 & 0
   \end{pmatrix}\right)\mathrm{SL}\left(2,\mathbb{R}\right)\right)$. $\Gamma\left(W\right)$ 
   has a semigroup structure as sub-semigroup of $SL\left(2,\mathbb{C}\right)$. We denote the universal covering of $\Gamma\left(W\right)$ as $\widetilde{\Gamma\left(W\right)}$. This domain is called Olshanski semigroup. We define $\gamma_{z} := \mathrm{Exp}\left(iz\begin{pmatrix}
      0 & 1 \\
      -1 & 0
   \end{pmatrix}\right)$.

   \begin{fact} [{\cite[Section3.8]{MR2956043}}]
      
      Suppose $a>\max\left\{0, 2-N\right\}$.
      \begin{enumerate}
         
         \item $\Omega_{a}$ extends to strongly continuous representation of semigroup
            $$ \Omega_{a}: \widetilde{\Gamma\left(W\right)} \longrightarrow \mathscr{B}\left(L^2\left(\mathbb{R}^{N}\backslash\{0\},\frac{dx}{|x|^{2-a}}\right)\right).$$
         \item For any $f \in L^2\left(\mathbb{R}^{N}\backslash\{0\},\frac{dx}{|x|^{2-a}}\right)$, the functions
            $$ \gamma \mapsto \left(\Omega_{a}(\gamma)f,f\right) \hspace{10pt} \widetilde{\Gamma\left(W\right)} \longrightarrow \mathbb{C}$$
            are analytic on interia points of $\widetilde{\Gamma\left(W\right)}$.
         \item If $\mathrm{Re}\left(z\right) >0$, then $\Omega_{a}\left(\gamma_{z}\right)$ are Hilbert--Schmidt  Operators.
         \item If $\mathrm{Re}\left(z\right) = 0$, then $\Omega_{a}\left(\gamma_{z}\right)$ are unitary Operators.

      \end{enumerate}

   \end{fact}

   Now, we are ready to define the $a$-generalized Laguerre semigroup and the $a$-generalized Fourier transform.

   \begin{definition}
      
      Suppose $a>\max\left\{0, 2-N\right\}$. 

      We define the \textbf{$\bf{(0,a)}$-generalized Laguerre semigroup} by
      $$e^{\frac{z}{a}\left(|x|^{2-a}\Delta-|x|^{a}\right)} :=\Omega_{a}\left(\gamma_{z}\right) \hspace{20pt}\left(\mathrm{Re}\left(z\right) \geqq 0\right)$$

      We define the \textbf{$\bf{(0,a)}$-generalized Fourier transformation} $\mathscr{F}_{a}$ by 
      $$\mathscr{F}_{a} := e^{\frac{\pi i}{2}\left(\lambda_{a}+1\right)}\Omega_{a}\left(\gamma_{\frac{\pi i}{2} }\right)$$
   
   \end{definition}

   The $(0,a)$-generalized Laguerre semigroup and the $(0,a)$-generalized Fourier transform are bounded operators. 
   Thus, by the Schwartz kernel theorem, there exsit distribution kernels $\Lambda_{a}(x,y;z)$ and $B_{a}(x,y)$ such that
   
   \begin{equation}
      \Omega_{a}\left(\gamma_{z}\right) f(x) = c_{a}\int_{\mathbb{R}^{N}}\Lambda_{a}(x,y;z)f(y)\frac{dy}{|y|^{2-a}}\label{intkers}
   \end{equation}

   \begin{equation}
      \mathscr{F}_{a} f(x) = c_{a}\int_{\mathbb{R}^{N}}B_{a}(x,y)f(y)\frac{dy}{|y|^{2-a}}\label{intfou}
   \end{equation}

   for all $f\in L^2\left(\mathbb{R}^{N}\backslash\{0\},\frac{dx}{|x|^{2-a}}\right)$. 
   
   Here, $c_{a}$ is a constant for normalization defined as 
   $$c_{a} :=\left(\int_{\mathbb{R}^{N}}e^{-\frac{1}{a}|x|^{a}}\frac{dx}{|x|^{2-a}}\right)^{-1} = \left(a^{\lambda_{a}}\Gamma\left(\lambda_{a}+1\right)\mathrm{vol}\left(S^{N-1}\right)\right)^{-1}$$

   \vspace{15pt}
   Now, we explore the concrete formulas of $\Lambda_{a}(x,y;z)$ and $B_{a}(x,y)$. In \cite{MR2956043}, it is proved that $\Lambda_{a}(x,y;z)$ and $B_{a}(x,y)$ are in fact not only distributions but also functions.

   \vspace{5pt}
   Via the expansion $$L^{2}\left(\mathbb{R}^{N},\frac{dx}{|x|^{2-a}}\right)=\sideset{}{^\oplus}\sum_{m\in\mathbb{N}}\mathscr{H}^{m}\left.\left(\mathbb{R}^{N}\right)\right|_{S^{N-1}}\otimes L^{2}\left(\mathbb{R}_{+},r^{N+a-2}\frac{dr}{r}\right),$$
   we may consider the decomposition 
   $$ \Omega_{a} = \sideset{}{^\oplus}\sum_{m\in\mathbb{N}} \mathrm{id}_{|\mathscr{H}^{m}(\mathbb{R}^{N})}\otimes \Omega_{a}^{(m)}$$

   and may consider the distribution kernels $\Lambda^{(m)}_{a}$ such that
   \begin{equation}\Omega^{(m)}_{a}(\gamma_{z}) f(r) = \int_{\mathbb{R}^{N}}\Lambda^{(m)}_{a}(r,s;z)f(s)s^{N+a-2}\frac{ds}{s} \label{intker} \end{equation}
   for all $f\in L^{2}\left(\mathbb{R}_{+},r^{N+a-2}\frac{dr}{r}\right)$.

   \begin{fact}[{\cite[Section4.1]{MR2956043}}]\label{radint}
      
      Assume $a>\max\left\{0, 2-2m-N\right\}$, then for $ z \in \mathbb{C}_{+}\backslash i \mathbb{R}$,
      $$\Lambda^{(m)}_{a}(r,s;z) = \frac{(rs)^{m}}{a^{\lambda_{a,m}}\mathrm{sinh}(z)^{\lambda_{a,m}+1}}e^{-\frac{r^{a}+s^{a}}{a}\mathrm{coth}(z)}\tilde{I}_{\lambda_{a,m}}\left(\frac{2}{a}\frac{(rs)^{a/2}}{\mathrm{sinh}(z)}\right).$$

   \end{fact}

   In proof, we write the integral kernels by using a basis which includes Laguerre polynomials and apply the equality called Hille--Hardy's formula.

   \vspace{10pt}
   Now, we would like to sum up these results.
   We define
   \begin{equation}
 	   \mathscr{I}(b,\nu, w, t) : = \Gamma(b\nu +1)
	   \sum_{m=0}^{\infty} \left(\frac{w}{2}\right)^{bm} \tilde{I}_{b(m+ \nu)}(w)\check{C}_{m}^{\nu}(t).\label{I}
   \end{equation}
   Here, $\tilde{I}_{\alpha}$ is a normalized I-Bessel functions and $\check{C}_{m}^{\nu}(t)$ is normalized Gegenbauer polynomials (For definition, see Appendix \ref{App}). 
   $\mathscr{I}(b,\nu, w, t)$ converges absolutely and uniformly on comapct subsets of 
	$ U:=\{(b.\nu,w,t)\in \mathbb{R}_{+}\times\mathbb{R}\times\mathbb{C}\times[-1,1] : 1+b\nu >0\}.$
	 \cite[Section4.3]{MR2956043}

   \begin{fact}[{\cite[Section4.3]{MR2956043}}]\label{main2}
      Assume $a>\max\left\{0, 2-N\right\}$. 

      \begin{enumerate}
         \item
            For $ z \in \mathbb{C}_{+}\backslash i \mathbb{R}$, \hspace{4pt} $r,s \in \mathbb{R}_{\geq 0}$,\hspace{3pt}  $\omega, \mu \in S^{N-1}$
	         $$\Lambda_{a}(r\omega,s\mu;z)=\frac{1}{\mathrm{sinh}(z)^{\lambda_{a}+1}} e^{-\frac{r^a+s^a}{a}\mathrm{coth}(z)} 
	         \mathscr{I}\left(\frac{2}{a},\frac{N-2}{2},\frac{2(rs)^{a/2}}{a\,\mathrm{sinh}(z)}, \langle\omega,\mu\rangle\right)$$
         \item For $r,s \in \mathrm{R}_{\geq 0}$,\hspace{3pt}  $\omega, \mu \in S^{N-1}$,
            $$ B_{a}(r\omega,s\mu):= \mathscr{I}\left(\frac{2}{a},\frac{N-2}{2},\frac{2(rs)^{a/2}}{ai}, \langle\omega,\mu\rangle\right)$$
      \end{enumerate}	
	
   \end{fact}
   Here, normalized I-Bessel functions appear from Fact \ref{radint} and Gegenbauer polynomials appear as the integral kernels of the projection operators from $L^{2}(S^{N-1}$ to $\mathscr{H}^{m}(\mathbb{R}^{N})\left|_{S^{N-1}}\right.$.

   \vspace{10pt}

   The following is a property of $a$-generalized Fourier transform which is similar to ordinary Fourier analysis.
   \begin{fact}  [{\cite[Theorem 5.6, Theorem 5.7]{MR2956043}}] \label{inter}
      
      The unitary operator $\mathscr{F}_{a}$ satisfies the following intertwining relations:
      \begin{enumerate}
         \item $ \mathscr{F}_{a} \circ |x|^{a} = -|\xi|^{2-a}\Delta_{\xi} \circ \mathscr{F}_{a}$
         \item $ \mathscr{F}_{a} \circ |x|^{2-a}\Delta_{x} = -|\xi|^{a} \circ \mathscr{F}_{a}$.
      \end{enumerate}

      Equivalently, the distribution $B_{a}(x,y)$ solves the following differential equations:
      \begin{enumerate}
         \item $|x|^{a}B_{a}(x,\xi) = -|\xi|^{2-a}\Delta_{\xi} B_{a}(x,\xi)$
         \item $|x|^{2-a}\Delta_{x} B_{a}(x,\xi) = -|\xi|^{a}B_{a}(x,\xi)$.
      \end{enumerate}
   \end{fact}

   Lastly, we refer to the fact that $\mathscr{F}_{a}$ certainly generalizes the classical Fourier analysis.
   \begin{fact} [{\cite[Proposition 5.10]{MR2956043}}]

      When $a=2$,
      $$B_{a}(x,y)= e^{-i(x,y)}$$ 
      and $\mathscr{F}_{a}$ is the Eulidiean Fourier transform.
      \vspace{10pt}

      When $a=1$, 
      $$B_{a}(x,y)=\Gamma\left(\frac{N-1}{2}\right)J_{\frac{N-3}{2}}\left(\sqrt{2|x||y|+2(x,y)}\right)$$
      and $\mathscr{F}_{a}$ is the Hankel transform.

   \end{fact}

\section{Estimates of integral kernels}\label{section3}
	In this section, we estimate the rate of increase of integral kernels. These estimates are useful for justifying various discussions in $(0,a)$-generalized Fourier analysis. 

\vspace{10pt}

With reference to Fact \ref{main2}, estimating the function
\begin{equation}
 	\mathscr{I}(b,\nu, w, t) : = \Gamma(b\nu +1)
	\sum_{m=0}^{\infty} \left(\frac{w}{2}\right)^{bm} \tilde{I}_{b(m+ \nu)}(w)\check{C}_{m}^{\nu}(t)\label{I}
\end{equation}

on $ U:=\{(b,\nu,w,t)\in \mathbb{R}_{+}\times\mathbb{R}\times\mathbb{C}\times[-1,1] : 1+b\nu >0\} $ is our purpose of this section.
The following theorem is the key step for our estimate.

\begin{theorem}[Complex integral representation of $\mathscr{I}$]\label{main3}

The following Complex integral formula holds on $ U:=\{(b,\nu,w,t)\in \mathbb{R}_{+}\times\mathbb{R}\times\mathbb{C}\times[-1,1] : 1+b\nu >0\} $:
$$\mathscr{I}(b,\nu,w,t) = \frac{ \Gamma(b\nu +1 )}{2\pi i} \left(\frac{w}{2}\right)^{-b\nu}
\int_{C_{\epsilon,w}}z^{-b\nu}\frac{1-z^{-2b}}{\left(1-2tz^{-b}+z^{-2b}\right)^{\nu +1}}e^{\frac{w}{2}\left(z+\frac{1}{z}\right)}\frac{dz}{z}. $$

\vspace{5pt}
Here, the integral path $C_{\epsilon,w}$ ($\,\epsilon >0,\hspace{3pt}w=|w|e^{i\alpha}\in\mathbb{C}\,$) is sum of the following three paths $C_{1}, C_{2}, C_{3}$: \par

    \begin{itemize}
	    \item[$\left(C_{1}\right)$] The half straight line:  $\left(+\infty\right) e^{\sqrt{-1}\left(-\alpha-\pi\right)} \rightarrow \left(1+\epsilon\right) e^{\sqrt{-1}(-\alpha-\pi)}$
	    \item[($C_{2}$)] The counterclockwise circle: $\left(1+ \epsilon\right)e^{\sqrt{-1}(-\alpha-\pi)} \rightarrow \left(1+\epsilon\right) e^{\sqrt{-1}(-\alpha+\pi)}$
	    \item[($C_{3}$)] The half straight line:  $\left(1+\epsilon\right) e^{\sqrt{-1}(-\alpha+\pi)} \rightarrow \left(+\infty\right)e^{\sqrt{-1}(-\alpha+\pi)}$
    \end{itemize}

\begin{tikzpicture}
    \draw[->,>=stealth,semithick] (-1.65,0)--(1.65,0)node[above]{}; 
    \draw[->,>=stealth,semithick] (0,-1.65)--(0,1.65)node[right]{}; 
    \draw (0,0)node[below  left]{}; 
    \draw[very thick](-2.2053,1.1726)--(155:1)arc(155:505:1)--(-2.1182,1.3236);
    \fill (30:2) circle (1.5pt);
    \draw (30:2)node[right]{$w$};
    
    \coordinate[label=0:$w$] (A) at (30:2) {};
    \coordinate (B) at (0,0) {};
    \coordinate (C) at (1,0) {};
    \path[draw] (A) -- (B) -- (C) ;
    \path pic["$\alpha$",draw,angle radius=5mm,angle eccentricity=1.4] {angle = C--B--A};

    \coordinate (A2) at (150:2.5) {};
    \coordinate (B2) at (0,0) {};
    \coordinate (C2) at (-1,0) {};
    \path[draw] (C2) -- (B2) -- (A2) ;
    \path pic["$\alpha$",draw,angle radius=5mm,angle eccentricity=1.4] {angle = A2--B2--C2};
	
    \coordinate (E) at (-1.5125,0.7726); 
    \coordinate (F) at (-1.6625,0.8128); 
    \coordinate (G) at (-1.6223,0.8824); 
    \fill[black] (E)--(F)--(G)--cycle; 
    \draw[very  thick] (E)--(F)--(G)--cycle; 
	\draw (-1.6625,0.8128)node[below]{$C_{1}$}; 

	\coordinate (E2) at (-1.5553,0.9986); 
    \coordinate (F2) at (-1.4053,0.9584); 
    \coordinate (G2) at (-1.4455,0.8888); 
    \fill[black] (E2)--(F2)--(G2)--cycle; 
    \draw[very  thick] (E2)--(F2)--(G2)--cycle; 
	\draw (-1.4053,0.9584)node[above]{$C_{3}$}; 

    \coordinate (E3) at (0.7660,-0.6428); 
    \coordinate (F3) at (0.6388,-0.7319); 
    \coordinate (G3) at (0.7004,-0.7835); 
    \fill[black] (E3)--(F3)--(G3)--cycle; 
    \draw[very  thick] (E3)--(F3)--(G3)--cycle; 
	\draw (0.7004,-0.7835)node[right]{$C_{2}$}; 

    \draw[thin,<->] (240:0.05) -- (240:0.95);
	\draw (240:0.5)node[left]{$1+\epsilon$};
\end{tikzpicture}

\end{theorem}

For proof of Theorem \ref{main3}, we need two facts about special functions.

\begin{pfact}\label{lem bessel}
    For $\alpha \in \mathbb{C},\, w\in \mathbb{C^{\times}}$, 
    $$\tilde{I}_{\nu}(w) = \frac{1}{2\pi i} \int _{C_{\epsilon,w}}z^{-\nu}e^{\frac{w}{2}\left(z+\frac{1}{z}\right)}\frac{dz}{z}$$
    Here, the integral path $C_{\epsilon,w}$ is the same as in Theorem \ref{main3}.
\end{pfact}

    \begin{pfact}\label{lem geg}
    Suppose $\nu \in \mathbb{R}$ and $0 < \epsilon < 1$. \par\noindent
    For $|t|\leqq1$ and $ |x|<1-\epsilon$, uniformly,
    $$\frac{1-x^2}{\left(1-2tx+x^2\right)^{\nu+1} }=\sum_{m=0}^{\infty}\check{C}_{m}^{\nu}(t)x^{m}.$$
\end{pfact}

For proof of Fact \ref{lem bessel} and \ref{lem geg}, see Appendix \ref{App}.

\begin{proof}[Proof of \textup{Theorem \ref{main3}}]

    \begin{align} 
	    \MoveEqLeft\mathscr{I}(b,\nu,w,t) =  \Gamma(b\nu +1)\left(\frac{w}{2}\right)^{-b\nu}\sum_{m=0}^{\infty} \tilde{I}_{\nu(b+m)}(w)\check{C}_{m}^{\nu}(t) \\
	    &=  \Gamma(b\nu +1)\left(\frac{w}{2}\right)^{-b\nu}\sum_{m=0}^{\infty}\frac{1}{2\pi i} \int_{C_{\epsilon,w}}z^{-b(\nu + m)}e^{\frac{w}{2}(z+\frac{1}{z})}\frac{dz}{z}\times \check{C}_{m}^{\nu}(t) \\
	    &= \frac{ \Gamma(b\nu +1 )}{2\pi i} \left(\frac{w}{2}\right)^{-b\nu}\int_{C_{\epsilon,w}}z^{-b\nu}\sum_{m=0}^{\infty}z^{-bm}\check{C}_{m}^{\nu}(t)e^{\frac{w}{2}(z+\frac{1}{z})}\frac{dz}{z}  \\
	    &= \frac{ \Gamma(b\nu +1 )}{2\pi i} \left(\frac{w}{2}\right)^{-b\nu}\int_{C_{\epsilon,w}}z^{-b\nu}\frac{1-z^{-2b}}{(1-2tz^{-b}+z^{-2b})^{\nu +1}}e^{\frac{w}{2}(z+\frac{1}{z})}\frac{dz}{z} \label{gegapply}
    \end{align}
    \vspace{5pt}

    At the second equality, Fact \ref{lem bessel} is applied and at the fourth equality, 
    Fact \ref{lem geg} is applied. At the third equality, the exchange of integral and infinite sum is carried out. 
    This is justified by the Lebesgue convergence theorem as follows. 
    Because of the following inequality (For proof, see Appendix \ref{App}):
    \begin{align}
        \underset{-1\leqq t \leqq 1}{\mathrm{sup}}\left| \check{C}_{m}^{\nu}(t)\right| & \leqq {}^\exists B(\nu)m^{2\nu},
    \end{align}
    we are able to estimate as

    \begin{align} 
	    \MoveEqLeft \left|z^{-b\nu}\sum_{m=0}^{\infty}z^{-bm}\check{C}_{m}^{\nu}(t)e^{\frac{w}{2}(z+\frac{1}{z})}\right| \\ 
	    &\leqq B(\nu) \left( \sum_{m=0}^{\infty}  m^{2\nu}(1+\epsilon)^{-b(m+\nu)} \right) e^{\mathrm{Re}\left(\frac{w}{2}\left(z+\frac{1}{z}\right)\right)}. \label{upper1}
    \end{align}

    Since, 
    \begin{align}
        \mathrm{Re}\left(\frac{w}{2}\left(z+\frac{1}{z}\right)\right) & \leqq \left\{
        \begin{array}{ll}
            -|w| \left(|z|-|z|^{-1}\right)/2 & \left(z \in C_{1}, C_{3}\right) \\
            \mathrm{Re(w)}\mathrm{cos}\,\theta + \epsilon|w|& \left(z \in C_{2},\, \theta = \mathrm{arg}(z)\right),
        \end{array}
        \right. \label{upparg}
    \end{align}

    the right hand side of \eqref{upper1} is integrable about $z$, 
    so we may apply the Lebesgue's theorem.

\end{proof}

\vspace{10pt}
By using the formula in Theorem \ref{main3}, we obtain the following estimate of $\mathscr{I}$.

\begin{theorem}\label{Iest1}
    For any $k,l \in \mathbb{Z}_{\geqq 0}$,
    there exists a constant $C\left(\,=C_{k,l,b,\nu}\right) >0$ such that,
    $$\left|\frac{\partial ^{k}}{\partial w^{k}}\frac{\partial ^{l}}{\partial t^{l}}\mathscr{I}(b,\nu,w,t)\right|\leqq C\,|w|^{(2-b)\nu +2(l+1)}\,e^{|\mathrm{Re}(w)|}$$

    when $|w|$ is sufficiently large.

\end{theorem}

\begin{proof}\quad\par\vspace{3pt}

    Exchanging the integral and differentials, we may write
    \begin{align}
        \MoveEqLeft\frac{\partial ^{k}}{\partial w^{k}}\frac{\partial ^{l}}{\partial t^{l}}\mathscr{I}(b,\nu,w,t) \\
        & = \sum_{i=0}^{k}\left(\frac{w}{2}\right)^{-(b\nu + i)}\int_{C_{\epsilon,w}}\frac{z^{b\nu} L_{i}(z^{b},z)}{(1-2tz^{-b}+z^{-2b})^{\nu+ l +1}}e^{\frac{w}{2}(z+\frac{1}{z})}dz \\
    \end{align}
    Here, $L_{i}(-,-)$ are Laurent polynomials.

    \vspace{5pt}
    On $C_{\epsilon,w}$,
    \begin{align}
        \MoveEqLeft |1-2tz^{-b}+z^{-2b}|  \\
        & = |z^{-b}-\alpha||z^{-b}-\bar{\alpha}| \\
        &\geqq (1-(1+\epsilon)^{-b})^2 \geqq (b\epsilon)^2
    \end{align}

    Hence,

    \begin{align}
	    \MoveEqLeft \left|\left(\frac{w}{2}\right)^{-(b\nu+ i)}\int_{C_{\epsilon,w}}\frac{z^{b\nu}L_{i}(z^{b},z)}{(1-2tz^{-b}+z^{-2b})^{\nu+l +1}}e^{\frac{w}{2}\left(z+\frac{1}{z}\right)}dz\right|\\
	    & \leqq {}^\exists L_{0}\,\epsilon^{-2(\nu + l + 1)} \left| w \right|^{-(b\nu+i)}\int_{C_{\epsilon,w}} |z|^{{}^\exists N}e^{\mathrm{Re}\left(\frac{w}{2}\left(z+\frac{1}{z}\right)\right)} |dz| \\
	    & = L_{0}\, \epsilon^{-2(\nu + l + 1)} \left| w \right|^{-(b\nu+i)}\left(\int_{C_{1}\cup C_{3}} |z|^{N}e^{\mathrm{Re}\left(\frac{w}{2}(z+\frac{1}{z})\right)} |dz|  +\int_{C_{2}}|z|^{N}e^{\mathrm{Re}\left(\frac{w}{2}\left(z+\frac{1}{z}\right)\right)} |dz| \right) \\
	    &\leqq L_{0}\, \epsilon^{-2(\nu+l+1)} \left| w \right|^{-(b\nu+i)}\left(2\int_{1}^{\infty} r^{N}e^{-|w|\frac{r-r^{-1}}{2}}dr +(1+\epsilon)^{N} \int_{0}^{2\pi}e^{\mathrm{Re(w)}\mathrm{cos}\,\theta + \epsilon |w|}d\theta \right)   \\
	    &\leqq {}^\exists L_{1}\,\epsilon^{-2(\nu+l+1)} \left|w\right|^{-(b\nu+i)}\left(1 + e^{|\mathrm{Re(w)}| + \epsilon |w|}\right)  \hspace{20pt}\\
	    &(\,\text{At the last inequarity, we assume } |w| \gg 1 \text{ and } 0< \epsilon \ll 1 \text{.}\,)\\
    \end{align}

    We may substitute $\epsilon = \frac{1}{|w|}$. Then, we obtaion

    $$\left|\frac{\partial ^{k}}{\partial w^{k}}\frac{\partial ^{l}}{\partial t^{l}}\mathscr{I}(b,\nu,w,t)\right|\leqq {}^\exists L_{2}\,|w|^{(2-b)\nu +2(l+1)}e^{|\mathrm{Re}(w)|}$$

    when  $|w|$ is sufficiently large.

\end{proof}

We also need to know the behavior of the integral kernels of Laguerre semigroup around the origin.\par

\begin{theorem}\label{Iest2}
    For any $k,l \in \mathbb{Z}_{\geqq 0}$,
    there exists the constant $C\left(\,=C_{k,l,b,\nu}\right) >0$ such that,
    $$\left|\frac{\partial ^{k}}{\partial w^{k}}\frac{\partial ^{l}}{\partial t^{l}}\mathscr{I}(b,\nu,w,t)\right|\leqq C\,|w|^{-k}$$

    when $|w|$ is sufficiently small.

\end{theorem}

\begin{proof}\quad\par\vspace{3pt}

    We consider the term by term differntiations of
    \begin{equation}
        \mathscr{I}(b,\nu, w, t) : = \Gamma(b\nu +1)
        \sum_{m=0}^{\infty} \left(\frac{w}{2}\right)^{bm} \tilde{I}_{b(m+ \nu)}(w)\check{C}_{m}^{\nu}(t)\label{I}\,.
    \end{equation}

    In particular, we need to show the uniform convergences on compacts of 
    $$ \sum_{m=0}^{\infty} w^{k}\frac{\partial ^{k}}{\partial w^{k}}(\left(\frac{w}{2}\right)^{bm} \tilde{I}_{b(m+ \nu)}(w))\frac{\partial ^{n}}{\partial t^{n}}(\check{C}_{m}^{\nu}(t))\,.$$

    By using the formulas and the estimates (For proof, see Appendix \ref{App}),
    $$ \frac{d}{dt}\check{C}^{\nu}_{n}(t) = 2(\nu+1)C^{\nu + 1}_{n-1}(t), $$
    $$ \frac{d}{dw}\tilde{I}_{\nu}(w) = \frac{w}{2}\tilde{I}_{\nu+1}(w), $$
    $$ \underset{-1\leqq t \leqq 1}{\mathrm{sup}}\left| \check{C}_{m}^{\nu}(t)\right|  \leqq {}^\exists B(\nu)m^{2\nu}, $$
    $$ \left|\tilde{I}_{\nu}(w)\right| \leqq \frac{e^{Re(w)}}{\Gamma(\nu +1)} \hspace{30pt} (\nu \geqq -1/2),$$

    we obtain the results.

\end{proof}

By the above theorems, we obtain estimates of the integral kernels which are useful in $(0,a)$-generalized Fourier analysis.

\begin{corollary}\label{3cor}

    Let $B_{a}(x,y)$ be the integral kernel of $(0,a)$-generalized Fourier transformation. 
    Then the following estimates hold:
    \vspace{5pt}
    
	For any $\alpha,\beta\in \mathbb{N}^{d}$ , there exists a constant $C\left(\,=C_{a,\alpha,\beta}\right)>0$ such that,
	$$\left| \frac{\partial ^{\alpha}}{\partial x^{\alpha}}\frac{\partial ^{\beta}}{\partial y^{\beta}} B_{a}(r\omega,s\mu )\right| \leqq C \left(1+ r^{-|\alpha|} + s^{-|\beta|} + (rs)^{ M+a\left(|\alpha|+|\beta|\right) } \right).$$	
    Here, $ M := \left(a-1\right)\frac{N-2}{2}+a$.

\end{corollary}

\textbf{Remark)}\par
Bounds for $\mathscr{I}(b,\nu,w,t)$ are also studied more accurately by \cite[Theorem 1.4]{MR4871318}, independently of Corollary \ref{Iest1}.
They show that; \par\vspace{5pt}
For $b >0$, there exists $C_{b,\nu}>0$ such that
$$\left|\mathscr{I}(b,\nu,w,t)\right| \leqq C_{b,\nu}(1+|w|^{(2-b)\nu})$$
and for $b > 2$, there exists $C_{b,\nu}>0$ such that
$$\left|\mathscr{I}(b,\nu,w,t)\right| \leqq C_{b,\nu}(1+|w|^{(1-b)\nu}).$$
\vspace{3pt}

As for bounds for $B_{a}(x,y)$, there is also the study of \cite[Theorem 3.3]{MR4832104}.
They show that; \par\vspace{5pt}
For $ a = p/q \in \mathbb{Q}_{+} $, $ N \in 2\mathbb{N} $, there exists $C_{p/q} >0$ such that 
$$ |B_{p/q}(x,y)| \leqq C_{p/q}\left(1+|x||y|\right)^{\frac{3pN}{2}+N-2} $$
by using another methods, though this is weaker result than Corollary \ref{3cor} and \cite[Theorem 1.4]{MR4871318}.

\vspace{7pt}
In \cite{MR4832104}, they also study the case of $N \in 2\mathbb{N}$, $a=4$ or $N=2$, $a=2^{l}/q$ by concrete calculation, 
and they show another proof of the result of \cite[Theorem 1.4]{MR4871318} for this case.
Additionally, they show that when $N \in 2\mathbb{N}$, $a=4$, the estimate obtained there;
$$ |B_{4}(x,y)| \leqq C(1+|x||y|)^{\frac{N-2}{2}}$$
is optimal with respect to degree of polynomial.

\vspace{7pt}
In the case when $a=2/m$, $m\in\mathbb{N}$, the estimates  
$$ |B_{2/m}(x,y)| \leqq 1 $$
 are also known.
They follow from the Laplace transform of $\mathscr{I}(m,\nu,xw,t)$ with respect to $x >0 $ which appears in \cite{MR3759078}, \cite{MR3767365}.

\vspace{15pt}

\begin{proof}[Proof of \textup{Corollary \ref{3cor}}]\quad\par\vspace{3pt}
    Recall that, by Theorem \ref{main2},
    \begin{align}
         B_{a}(r\omega,s\mu) = \mathscr{I}\left(\frac{2}{a},\frac{N-2}{2},\frac{2(rs)^{a/2}}{ai}, \langle\omega,\mu\rangle\right).
    \end{align}

    Considering the coordinate transformation 
    $$\frac{\partial}{\partial x_{i}} = \frac{x_{i}}{r} \frac{\partial}{\partial r} + \frac{1}{r} W \hspace{30pt} ( \,\text{Here, }W \text{ is vector field of } S^{N-1}) $$
    and using Theorem \ref{Iest1} and Theorem \ref{Iest2}, we obtain the results.
\end{proof}

\vspace{20pt}
By the same method, we are also able to estimate the integral kernel of Laguerre semigroup and its differentials.
In this article, we see the estimate of itself.

\begin{corollary}
    
    There exsits $C\left(\,=C_{a}\right) >0 $ such that,
    \begin{align}
        \MoveEqLeft \left|\Lambda_{a}(r\omega,s\mu;z)\right| \\
        & \hspace{-0pt}\leqq \frac{C}{\mathrm{sinh}(z)^{\lambda_{a}+1}} \left(1 + \left|\frac{(rs)^{a/2}}{a\mathrm{sinh}(z)}\right|^{\frac{2M}{a}}\right)e^{-\frac{1}{a}\mathrm{Re}\left( \frac{r^{a}+s^{a}}{\mathrm{tanh}(z)} - \frac{2(rs)^{a/2}}{\mathrm{sinh}(z)}\right)}. \\
    \end{align}
    Here, $ M := \left(a-1\right)\frac{N-2}{2}+a$.
\end{corollary}

\begin{proof}\quad\par\vspace{3pt}

    Apply Theorem \ref{Iest1} to

    \begin{align}
        \Lambda_{a}(r\omega,s\mu;z) & =\frac{1}{\mathrm{sinh}(z)^{\lambda_{a}+1}} e^{-\frac{r^a+s^a}{a}\mathrm{coth}(z)} 
       \mathscr{I}\left(\frac{2}{a},\frac{N-2}{2},\frac{2(rs)^{a/2}}{a\,\mathrm{sinh}(z)}, \langle\omega,\mu\rangle\right).\\
   \end{align}

\end{proof}

\vspace{7pt}

\section{$a$-deformed Heat Theory}\label{section4}
	In this section, 
we define the $a$-deformed heat kernel 
and explore its basic properties.

\subsection{$a$-deformed heat kernel}\label{section4_1}

\begin{defprop}[$a$-deformed heat kerenel]\label{heatdef}
	
	Suppose $a>\max\left\{0, 2-N\right\}$. Then the following integral converges absolutely and uniformly on compacts of $\{(x,y,t)\in \mathbb{R}^{n} \times  \mathbb{R}^{n} \times \mathbb{R}_{>0}\}$:
		\begin{equation}\label{aheat}
			h_{a}(x,y;t) = c_{a}\int_{\mathbb{R}^{n}}\overline{B(x,\xi)}B(y, \xi)e^{-\frac {t}{a}|\xi|^{a}}\frac{d\xi}{|\xi|^{2-a}}.
		\end{equation}
	we define $h_{a}(x,y;t)$ as the \textbf{$a$-deformed heat kernel}.

\end{defprop}

\begin{proof}\quad\par\vspace{3pt}

	Let $K$ be a compact subset of $\{(x,y,t)\in \mathbb{R}^{n} \times  \mathbb{R}^{n} \times \mathbb{R}_{>0}\}$. 
	There exsist positive numbers $A, \epsilon$ such that $K \subset \{(x,y,t)\,;\, |x| <A,\,|y|<A,\, \epsilon < t\}$.
	Then, with reference to Corollary \ref{3cor}, the absolute value of the integland is bounded from the above as
	\begin{align}
		\MoveEqLeft \left|\overline{B(x,\xi)}B(y, \xi)e^{-\frac {t}{a}|\xi|^{a}}\frac{d\xi}{|\xi|^{2-a}}\right| \leqq {}^\exists C(1+A|\xi|)^{2M}e^{-\frac{\epsilon}{a}|\xi|^{a}}\frac{d\xi}{|\xi|^{2-a}}. & \\
	\end{align} 
	Since, in polar coordinate,
	$$\frac{d\xi}{|\xi|^{2-a}} = r^{N+a-2}\frac{dr}{r}d\omega, $$
	the right hand side of the inequality is integrable and this proves the claim.

\end{proof}

\begin{proposition}\label{heatexc}
	
	The $a$-deformed heat kernel $h_{a}(x,y;t)$ is $\mathrm{C}^{\infty}$-class function on $\{(x,y,t)\in \mathbb{R}^{n}\backslash\{0\} \times  \mathbb{R}^{n}\backslash\{0\} \times \mathbb{R}_{>0}\}$ 
	and for any $\alpha, \beta \in \mathbb{N}^{d}, \hspace{5pt}l\in \mathbb{N}$
		\begin{equation}\label{adheat}
			\frac{\partial ^{|\alpha|+|\beta|+l} h_{a}}{\partial x^{\alpha}\partial y^{\beta}\partial y^{l}}(x,y;t) = c_{a}\int_{\mathbb{R}^{n}}\frac{\partial^{\alpha} \overline{B}}{\partial x^{\alpha}}(x,\xi)\frac{\partial ^{\beta}B}{\partial y^{\beta}}(y, \xi)\left(-\frac{|\xi|^{a}}{a}\right)^{l}e^{-\frac {t}{a}|\xi|^{a}}\frac{d\xi}{|\xi|^{2-a}}
		\end{equation}

\end{proposition}

\begin{proof}\quad\par\vspace{3pt}

	By similar estimates as Definition-Proposition \ref{heatdef}, 
	we are able to prove the right hand side of the equality converges absolutely and uniformly on compacts of
	$\{(x,y,t)\in \mathbb{R}^{n}\backslash\{0\} \times  \mathbb{R}^{n}\backslash\{0\} \times \mathbb{R}_{>0}\}$.
	Then we may exchange the integral and the differentials, and we obtain the equality.

\end{proof}

\begin{corollary}\label{heatkereq}

	The $a$-deformed heat kernel satisfies the $a$-generalized heat equation:

	$$ \left(\partial_{t} - \frac{1}{a}|x|^{2-a}\Delta \right)h_{a}(x,y;t) =0.$$

\end{corollary}

\begin{proof}\quad\par\vspace{3pt}

	The result follows from Fact \ref{inter} and Theorem \ref{heatexc}.

\end{proof}

\vspace{15pt}

	\subsection{Description of $a$-deformed heat kernel by $\mathscr{I}$}\label{section4_2}

Recall that, the function $\mathscr{I}$ is defined as,
\begin{equation}
 	\mathscr{I}(b,\nu, w, t) : = \Gamma(b\nu +1) 
	\sum_{m=0}^{\infty} \left(\frac{w}{2}\right)^{bm}\tilde{I}_{b(m+\nu)}(w)\check{C}_{m}^{\nu}(t)\label{4I}
\end{equation}
  
Now, we relate $\mathscr{I}$ with the $a$-deformed heat kernel by the following theorem. 
 
\begin{theorem}[Expansion formula] \label{4.1thm}
	Suppose $a>\max\left\{0, 2-N\right\}$. Let $h_{a}(x,y;t)$ be the $a$-deformed heat kernel.
 	Then, for $r,s \in \mathbb{R}_{\geq 0} \hspace{7pt} \omega, \mu \in S^{N-1}$,
 	\begin{equation}
 		h_{a}(r\omega,s\mu;t) = t^{-\left(\lambda_{a}+1\right)}
		e^{-\frac{r^{a}+s^{a}}{at}}\times\mathscr{I}
		\left(\frac{2}{a},\frac{N-2}{2},\frac{2(rs)^{a/2}}{at},\langle\omega,\mu\rangle\right).
	\end{equation}
\end{theorem}
 
We need two facts for the proof of Theorem \ref{4.1thm}. 
\vspace{5pt}

Consider the function: 
 
$$P_{m}(\omega,\mu) := \check{C}_{m}^{\nu}(\langle\omega,\mu\rangle) \hspace{20pt}(\,\omega,\mu\,\in\, S^{N-1})$$
 
Here $\nu=\frac{N-2}{2}$.
Then, the $(0,a)$-generalized Fourier integral kernel is written as follow (see Theorem \ref{main2}),
 
\begin{equation}
 B_{a}(x,y)=\Gamma(\lambda_{a}+1)\sum_{m=0}^{\infty}(ia)^{-\frac{2m}{a}}(rs)^{m}\tilde{J}_{\lambda_{a,m}}\left(\frac{2}{a}(rs)^{a/2}\right)P_{m}(\omega,\mu ) \,.\label{lagexp}
\end{equation}
 
$P_{m}$ has the following reproducing property (For proof, see Appendix \ref{App}).
 
\begin{pfact}[Reproducing property of $P_{m}$]\label{rep}
	For $ \,\omega,\mu\,\in\, S^{N-1} $,
 	\begin{align}
		\MoveEqLeft\frac{1}{\mathrm{vol}{(S^{N-1})}}\int_{S^{N-1}}P_{m}(\omega, \mu)P_{m'}(\omega, \eta)\,d\omega
	\,=\, \left\{
	\begin{array}{ll}
		P_{m}(\mu, \eta) & (\text{if}\quad m = m') \\
		0 & (\text{if}\quad m \neq m')
	\end{array}
	\right. 
	\end{align}
\end{pfact}

\vspace{15pt}
We also use the following formula called Weber's second exponential integral. 

\begin{pfact}[{Weber's second exponential integral {\cite[6.615]{MR2360010}}}]\label{wiber}

	Suppose $\alpha,\beta,\delta >0$ and $\nu > -1$, then

	$$\int_{0}^{\infty}e^{-\delta T}J_{\nu}(2\alpha\sqrt{T})J_{\nu}(2\beta\sqrt{T}) \,dT \,=\, \frac{1}{\delta}e^{-\frac{\alpha^2 + \beta^2}{\delta}}I_{\nu}(\frac{2\alpha\beta}{\delta})$$

\end{pfact}

This formula appear also in \cite{MR2956043}, as the corollary of the composition law of the Laguerre semigroup.

\vspace{10pt}
\begin{proof}[Proof of \textup{Theorem \ref{4.1thm}}]
	\begin{align}
		\MoveEqLeft h_{a}(r\omega,s\mu;t) & \\
		&= c_{a}\int_{\mathbb{R}^{N}}\overline{B(r\omega,\xi)}B(s\mu, \xi)e^{-\frac {t}{a}|\xi|^{a}}\frac{d\xi}{|\xi|^{2-a}} \\
		\begin{split}& =  c_{a}\,\Gamma(\lambda_{a}+1)^2\sum_{m,m'=0}^{\infty} \int_{u \in \mathbb{R}_{>0}}\int_{\eta \in S^{N-1}}
		 a^{-\frac{2}{a}(m+m')}\\
		 & \hspace{30pt}\times (ru)^{m}\tilde{J}_{\lambda_{a,m}}\left(\frac{2}{a}(ru)^{a/2}\right) \times (su)^{m'}\tilde{J}_{\lambda_{a,m'}}\left(\frac{2}{a}(su)^{a/2}\right)\\
		 &\hspace{50pt} \times P_{m}(\omega,\eta )P_{m'}(\mu,\eta )\,e^{-\frac{t}{a}u^{a}}u^{N+a-2}\frac{du}{u}d\eta\end{split}\label{h1}
	\end{align}
	
Here, we used the expansion \eqref{lagexp}, and carried out the exchange of the infinite sum and integral. 
The exchange is justified by the following two inequalities (For proof, see theorem \ref{App}).

\begin{align}
	\left|\tilde{J}_{\nu}(x)\right|\leqq \frac{1}{\Gamma(\nu+1)}\hspace{20pt}( \,\text{When }\,\nu \geqq -1/2\,)
\end{align}

\begin{equation}
	\left| P_{m}(\omega,\mu) \right| \leqq B\left(\frac{N-2}{2}\right)m^{N-2}
\end{equation}

\vspace{20pt}

Then, by the reproducing property (Fact \ref{rep}), 

\begin{align} 
	\begin{split}&  \eqref{h1} = c_{a}\,\Gamma(\lambda_{a}+1)^2\mathrm{vol}(S^{N-1})
		\sum_{m=0}^{\infty} \int_{0}^{\infty}
		a^{-\frac{2m}{a}\times 2}(ru)^{m}(su)^{m}\\
		& \hspace{30pt}\times \tilde{J}_{\lambda_{a,m}}\left(\frac{2}{a}(ru)^{a/2}\right)\tilde{J}_{\lambda_{a,m}}\left(\frac{2}{a}(su)^{a/2}\right) \\
		& \hspace{50pt} \times P_{m}(\omega,\mu ) 
		 \,e^{-\frac{t}{a}u^{a}}u^{N+a-2}\frac{du}{u}\, . \end{split}\label{h2}
\end{align}
	
	\vspace{10pt}
	
We change variables as $R = \frac{1}{a}r^{a}$ , $S = \frac{1}{a}s^{a}$ and $U = \frac{1}{a}u^{a}$, then
	
\begin{align}
		\begin{split}& \eqref{h2}=  c_{a}\,\Gamma(\lambda_{a}+1)^2 \mathrm{vol}(S^{N-1})a^{\lambda_{a}}\times \\
		&\qquad\sum_{m=0}^{\infty}\int_{0}^{\infty}
		\frac{J_{\lambda_{a,m}}\left(2\sqrt{RU}\right)}{R^{\lambda_{a}/2}} \frac{J_{\lambda_{a,m}}\left(2\sqrt{SU}\right)}{S^{\lambda_{a}/2}}  \,e^{-tU}\,dU \,\times P_{m}(\omega,\mu ) \end{split} \label{h3}
\end{align}

\vspace{10pt}
Finally, we apply the Wiber's second exponential integral (Fact \ref{wiber}), 
and organize the formula:
	\begin{align}
		& \eqref{h3} =  \Gamma(\lambda_{a}+1)\times
		\frac{e^{-\frac{R+S}{t}}}{t}\sum_{m=0}^{\infty}\frac{I_{\lambda_{a,m}}\left(\frac{2\sqrt{RS}}{t}\right)}{(RS)^{\lambda_{a}/2}} P_{m}(\omega,\mu ) \\
				& = t^{-(\lambda_{a}+1)}e^{-\frac{R+S}{t}}\times \mathscr{I}\left(\frac{2}{a}, \frac{N-2}{a}, \frac{2\sqrt{RS}}{t}, \langle\omega,\mu\rangle\right) \\
		& = t^{-(\lambda_{a}+1)}e^{-\frac{r^{a}+s^{a}}{at}}\times \mathscr{I}\left(\frac{2}{a}, \frac{N-2}{a}, \frac{2(rs)^{a/2}}{at}, \langle\omega,\mu\rangle\right).
\end{align}
\end{proof}

\begin{corollary}\label{4cor}
	Suppose $a>\max\left\{0, 2-N\right\}$. Let $h_{a}(x,y;t)$ be the $a$-deformed heat kernel then,

	\begin{enumerate}
			
		\item $h_{a}(x,y;t) $ is real valued and symmetric about $x,y$.
		\item There exists a constant $C(\,=C_{a}) >0$ such that
		 	\begin{align}
				\MoveEqLeft \left| h_{a}(x,y; t)\right| & \\
				& \leqq C \, t^{-(\lambda_{a}+1)}e^{-\frac{\left(r^{a/2}-s^{a/2}\right)^2}{at}}\left(1+\left(\frac{(rs)^{a/2}}{at}\right)^{\frac{2M}{a}}\right) \\
			\end{align}
	\end{enumerate}

\end{corollary}

\begin{corollary}\label{Gau}
	Suppose $a>\max\left\{0, 2-N\right\}$. Let $h_{a}(x,y;t)$ be the $a$-deformed heat kernel. 
	Then,
	$$ h_{a}(0,y;t)= t^{-(\lambda_{a}+1)}e^{-\frac{|y|^{a}}{at}}$$
\end{corollary}

\begin{corollary}
	Suppose $a>\max\left\{0, 2-N\right\}$. Let $\Lambda_{a}(x,y;z)$ be the integral kernel of the $(0,a)$-generalized Laguerre semigroup and $h_{a}(x,y;t)$ be the $a$-deformed heat kernel then,
	$$\Lambda_{a}(x,y;z) = e^{-\frac{|x|^{a}+|y|^{a}}{a}\left(\mathrm{coth}(z)-\frac{1}{\mathrm{sinh}(z)}\right)}
			h_{a}(x,y; \mathrm{sinh}(z))$$

\end{corollary}

\vspace{20pt}

We also able to derive the integral formulas which we use later.

\begin{corollary}[Total integral formula]\label{inteq}

	Suppose $a>\max\left\{0, 2-N\right\}$. Let $h_{a}(x,y;t)$ be the $a$-deformed heat kernel. Then,

			$$ c_{a}\int_{\mathbb{R}^{n}}h_{a}(x,y,t) \frac{dy}{|y|^{2-a}} =1$$

\end{corollary}

\begin{corollary}\label{intpoly}
	Suppose $a>\max\left\{0, 2-N\right\}$. Let $h_{a}(x,y;t)$ be the $a$-deformed heat kernel. 
	Then there exists a $m$th degree bivariate homogenous polynomial $f_{m}(T,X)$ such that:
			$$ c_{a}\int_{\mathbb{R}^{n}}h_{a}(x,y,t)|y|^{ma} \frac{dy}{|y|^{2-a}} =f_{m}(t,|x|^{a})$$
\end{corollary}

\begin{corollary}\label{intx}
	Suppose $a>\max\left\{0, 2-N\right\}$. 
	Let $h_{a}(x,y;t)$ be the $a$-deformed heat kernel. 
	Then,
	$$ c_{a}\int_{\mathbb{R}^{n}}h_{a}(0,y,t_{1})h_{a}(y,z,t_{2})\frac{dy}{|y|^{2-a}} 
				= h_{a}(0,z,t_{1}+t_{2}) $$
		
\end{corollary}

For proof, we need the following lemma.

\begin{plemma}\label{wiber2}
	For $\alpha,\delta >0$, $\nu > -1$ and $m \in \mathbb{N}$, then
	$$\int_{0}^{\infty}e^{-\delta T}\tilde{I}_{\nu}(2\alpha\sqrt{T}) T^{\nu+m} \,dT \,=\, \frac{1}{\delta^{\nu+m+1}}p_{m}\left(\frac{\alpha^2}{\delta}\right)e^{\frac{\alpha^2}{\delta}}$$
	Here, $p_{m}(x)$ is $m$th degree polynomial and $p_{0}(x)=1$.
\end{plemma}
\begin{proof}\quad\par\vspace{3pt}
	When $m=0$, we obtain the formula by considering the analytic continuation of the special case of Wiber's second exponential integral formula (Fact \ref{wiber}). 
	For $m>0$, we apply induction. 
	By differentiating both sides of equation by $\delta$, we obtain a recurrence formula
	\begin{equation}
		p_{m+1}(x) = xp_{m}(x) + xp'_{m}(x)+ (\nu +m+1)p_{m}(x). \label{rec}
	\end{equation}
	This shows the claim. \par
	The each steps of the proof are justified by the estimate of Bessel functions (Corollary \ref{ineqbessel} and Proposition \ref{esb2} in Appendix \ref{App}).
\end{proof}

\vspace{10pt}
\begin{proof}[Proof of \textup{Corollary \ref{inteq}} and \textup{Corollary \ref{intpoly}}]\quad\par\vspace{3pt}
	
			\begin{align}
				\MoveEqLeft c_{a}\int_{\mathbb{R}^{n}}h_{a}(x,y,t)|y|^{ma} \frac{dy}{|y|^{2-a}} &\\
				& = a^{-\lambda_{a}}t^{-(\lambda_{a}+1)}e^{-\frac{r^{a}}{at}} \int_{0}^{\infty} e^{-\frac{s^{a}}{at}}\tilde{I}_{\lambda_{a}}\left(\frac{2(rs)^{a/2}}{at}\right)s^{N+(m+1)a-2}\frac{ds}{s}\\
				& = a^{m}t^{-(\lambda_{a}+1)}e^{-\frac{R}{t}} \int_{0}^{\infty} e^{-\frac{S}{t}}\tilde{I}_{\lambda_{a}}\left(\frac{2\sqrt{RS}}{t}\right)S^{\lambda_{a}+m}\,dS\\
			\end{align}

			Here, $R = \frac{1}{a}r^{a}$ and $S = \frac{1}{a}s^{a}$.

			Applying the Lemma \ref{wiber2}, we obtain
			\begin{align}
				\MoveEqLeft c_{a}\int_{\mathbb{R}^{n}}h_{a}(x,y,t)|y|^{ma} \frac{dy}{|y|^{2-a}} \,=\, (at)^{m}p_{m}\left(\frac{R}{t}\right) \,=:\, f_{m}(t,|x|^{a})&\\
			\end{align}

\end{proof}
\begin{proof}[proof of \textup{Corollary \ref{intx}}]

		\begin{align}
			\MoveEqLeft c_{a}\int_{\mathbb{R}^{n}}h_{a}(0,y,t)h_{a}(y,z,s)\frac{dy}{|y|^{2-a}} &\\
			& = a^{-\lambda_{a}}(t_{1}t_{2})^{-(\lambda_{a}+1)} \int_{0}^{\infty} e^{-\frac{s^{a}}{at_{1}}}e^{-\frac{s^{a}+u^{a}}{at_{2}}}\tilde{I}_{\lambda_{a}}\left(\frac{2(su)^{a/2}}{at}\right)s^{N+a-2}\frac{ds}{s}\\
			& = (t_{1}t_{2})^{-(\lambda_{a}+1)} \int_{0}^{\infty} e^{-\frac{S}{t_{1}}}e^{-\frac{S+U}{t_{2}}}\tilde{I}_{\lambda_{a}}\left(\frac{2\sqrt{SU}}{t_{2}}\right)S^{\lambda_{a}}\,dS\\
		\end{align}

		Here, $S = \frac{1}{a}s^{a}$ and $U = \frac{1}{a}u^{a}$. Now, we apply the Lemma \ref{wiber2} then

		$$ c_{a}\int_{\mathbb{R}^{n}}h_{a}(0,y,t)h_{a}(y,z,s)\frac{dy}{|y|^{2-a}} = \frac{1}{(t_{1}+t_{2})^{\lambda_{a}+1}}e^{-\frac{|z|^{a}}{a(t_{1}+t_{2})}}$$

\end{proof}

\vspace{15pt}

	\subsection{$a$-deformed Heat flow}\label{section4_3}

\begin{defprop}[$a$-deformed heat flow]\label{defheatflow}
Suppose $a>\max\left\{0, 2-N\right\}$. Let $h_{a}(x,y;t)$ be the $a$-deformed heat kernel.
For a polynomially growth continuous function $u(y)$, 
	the following integral converges absolutely and uniformly on compacts of $\{(t,x)\in \mathbb{R}_{>0} \times \mathbb{R}^{n}\}$:
	
	\begin{equation}
		H_{t}u(x):= c_{a}\int_{\mathbb{R}^{N}}h_{a}(x,y;t)u(y)\frac{dy}{|y|^{2-a}}.
	\end{equation}
	We call $H_{t}u$ as $\bf{a}$\textbf{-deformed heat flow}.
	
\vspace{7pt}

\end{defprop}


\begin{proposition}\label{diffheatflow}
	$H_{t}u(x)$ is $\mathrm{C}^{\infty}$-class on $\{(t,x)\in \mathbb{R}_{>0} \times \mathbb{R}^{n}\}$ and
	$$ \left(\frac{\partial}{\partial t}\right)^{k} 
		\left(\frac{\partial}{\partial x}\right)^{l} H_{t}u(x) 
		= c_{a}\int_{\mathbb{R}^{N}}
		\left(\frac{\partial}{\partial t}\right)^{k} 
		\left(\frac{\partial}{\partial x}\right)^{l}
		h_{a}(x,y;t)u(y)\frac{dy}{|y|^{2-a}}.$$
	 
\end{proposition}
\begin{proof}[Proof of \textup{Definition-Proposition \ref{defheatflow}} and \textup{Proposition \ref{diffheatflow}}]\quad\par\vspace{3pt}
	The result follows from Theorems \ref{4.1thm}, \ref{Iest1} and \ref{Iest2} by similar estimates as Definition-Proposition \ref{heatdef}. 
	
\end{proof}

\begin{corollary}[Relation with $a$-deformed heat equation]\label{heatfloweq}
	Suppose $a>\max\left\{0, 2-N\right\}$. Let $H_{t}u$ be an $a$-deformed heat flow. $H_{t}u$ satisfies the $a$-deformed heat equation:
			\begin{equation}
				\left(\partial_{t}-\frac{1}{a}|x|^{2-a}\Delta\right)H_{t}u = 0
			\end{equation}
\end{corollary}
\begin{proof}\quad\par\vspace{3pt}
	The result follows immediately from Corollary \ref{heatkereq} and Proposition \ref{diffheatflow}.
\end{proof}

\begin{proposition}\label{heatflowfou}
	Suppose $a>\max\left\{0, 2-N\right\}$. Let $H_{t}u$ be an $a$-deformed heat flow. \par
	If $u \in C^{\infty}_{cpt}\left(\mathbb{R}^{N}\backslash\{0\}\right)$, 
			\begin{equation}
				H_{t}u(x)= c_{a}\int_{\mathbb{R}^{N}}\overline{B_{a}(x,\xi)}
				\mathcal{F}_{a}u(\xi)e^{-\frac{t}{a}|\xi|^{a}}\frac{d\xi}{|\xi|^{2-a}}
			\end{equation}
\end{proposition}

\begin{proof}\quad\par\vspace{3pt}
	\begin{align}
		\MoveEqLeft H_{t}u(x):= c_{a}\int_{\mathbb{R}^{N}}h_{a}(x,y;t)u(y)\frac{dy}{|y|^{2-a}} & \\
		& =c_{a}^2\int_{\mathbb{R}^{N}} \left(\int_{\mathbb{R}^{N}}
		\overline{B(x,\xi)}B(y, \xi)e^{-\frac {t}{a}||\xi||^{a}}\frac{d\xi}{|\xi|^{2-a}}\right) u(y)\frac{dy}{|y|^{2-a}} \\
		& = c_{a}^2\int_{\mathbb{R}^{N}} \overline{B(x,\xi)}\left(\int_{\mathbb{R}^{N}}
		B(y, \xi)u(y)\frac{dy}{|y|^{2-a}}\right) e^{-\frac {t}{a}||\xi||^{a}}\frac{d\xi}{|\xi|^{2-a}} \\
		& =c_{a}\int_{\mathbb{R}^{N}}\overline{B_{a}(x,\xi)}
			\mathcal{F}_{a}u(\xi)e^{-\frac{t}{a}|\xi|^{a}}\frac{d\xi}{|\xi|^{2-a}}
	\end{align}
	We have applied Fubini's theorem at the third equality.
		
\end{proof}

\begin{theorem}\label{basic}
	Suppose $a>\max\left\{0, 2-N\right\}$. Let $H_{t}u$ be an $a$-deformed heat flow. 
	If $u\in \mathrm{C}^{\infty}(\mathbb{R}^{N}\backslash\{0\})\,\cap \mathrm{C}\left(\mathbb{R}^{N}\right)$, then $H_{t}u$ satisfies the initial value condition: 
			\begin{equation}
				\lim_{t\rightarrow + 0} H_{t}u(x) = u(x)
			\end{equation}	
	
\end{theorem}
\textbf{Remark)}\par
This result is going to be improved in Corollary \ref{initial}.
\vspace{15pt}

\begin{proof}\quad\par\vspace{3pt}
	First, we assume that $ u(y)\in C^{\infty}_{\mathrm{cpt}}(\mathbb{R}^{N}\backslash\{0\})$. Then, by Proposition \ref{heatflowfou},
			\begin{align}
				\MoveEqLeft |H_{t}u(x)-u(x)| & \\
				& = \left|c_{a}\int_{\mathbb{R}^{N}}\overline{B_{a}(x,\xi)}
					\mathcal{F}_{a}u(\xi)
					\left(e^{-\frac{t}{a}|\xi|^{a}}-1\right)\frac{d\xi}{|\xi|^{2-a}}\right| \\
				& = \left|c_{a}\int_{\mathbb{R}^{N}}
					\frac{\overline{B_{a}(x,\xi)}}{\left(1+|\xi|^{a}\right)^{M+\lceil\frac{N+1}{a}\rceil}} 
					\left\{\left(1+ |\xi|^{a}\right)^{M+\lceil\frac{N+1}{a}\rceil}\mathcal{F}_{a}u(\xi)\right\} 
					\left(e^{-\frac{t}{a}|\xi|^{a}}-1\right)\frac{d\xi}{|\xi|^{2-a}}\right|\\
				& \longrightarrow 0 \qquad \text{(By Lebesgue Theorem).}
			\end{align}
			
			\vspace{10pt}
			Next, assume $u(y)$ be a general. Fix a point $x \in \mathbb{R}^{N}\backslash\{0\}$.
			By using a bump function, we may decompose $u(y)$ as
			\begin{equation}
			u(y)= u_{0}(y)+ u_{\mathrm{cpt}}(y) + u_{\infty}(y)
			\end{equation}
			
			Here, for example, \par
			\qquad $u_{0}(y)$ has support in  $\displaystyle\left\{|y| \leqq \frac{|x|}{2}\right\}$. \par
			\qquad $u_{\mathrm{cpt}}(y)$ has compact support and equal to $u$ around $x$. \par
			\qquad $u_{\infty}(y)$ has support in  $\left\{|y| \geqq 2|x|\right\}$. \par
			
			\vspace{10pt}
			By inequality (See Corollary \ref{4cor}), 
			 \begin{equation}
				\left| h_{a}(x,y; t)\right| \leqq Ct^{-(\lambda_{a}+1)}
				e^{-\frac{(r^{a/2}-s^{a/2})^2}{at}}\left(1+\left(\frac{(rs)^{a/2}}{at}\right)^{\frac{2M}{a}}\right)
			\end{equation}

			\vspace{5pt}
			$H_{t}u_{0}(x)$ and $H_{t}u_{\infty}(x)$ decrease exponentially when $t\rightarrow +0$. 
			Hence, 
			
			\begin{align}
				\MoveEqLeft \lim_{t\rightarrow 0}H_{t}u(x)  \\
				& = \lim_{t\rightarrow 0}H_{t}u_{\mathrm{cpt}}(x) \\
				& = u_{cpt}(x) \\
				& = u(x) 
			\end{align}
			
		At $x=0$, we are able to check directly, since
		$$ h_{a}(0,y;t) = t^{-(\lambda_{a}+1)}e^{-\frac{|y|^{a}}{at} } \hspace{20pt}(\text{Corollary \ref{Gau}}).$$

\end{proof}

	\subsection{Maximum principle and its corollaries}\label{section4_4}

\subsubsection{Maximum principle}

	Suppose $T$ and $L$ be positive and
	$$Q := [0,T]\times[-L,L]^{N}$$
	$$\Gamma := \left([0,T]\times (\{0\}\cup\partial[-L,L]^{N}) \right)
		\cup \left(\{0\}\times[-L,L]^{N}\right)\, .$$
		
\begin{theorem}[Maximum principle (Bounded ver.)]\label{maxb}
		Suppose $a>0$. Let $v = v(t,x)$ be a function, satisfying the following conditions:
		\begin{enumerate}
		
			\item $v$ is continuous on $Q$.
			
			\item $C^{1}$-class about $t$			
			\item	$C^{2}$-class about $x$ on $[-L,L]^{N}\backslash \{0\}$ .
				
			\item $\displaystyle v_{t}= \frac{1}{a}|x|^{2-a}\Delta v.$
						
\end{enumerate}
		
		\vspace{5pt}
		Then, 
			$$ \max_{(t,x) \in Q} v(t,x) \leqq \max_{(t,x)\in \Gamma} v(t,x)$$
	
\end{theorem}

\begin{proof}\quad\par\vspace{3pt}
		We define a function $w$ as $ w(t,x) := e^{-t} \left\{ v(t,x)- \max_{(t,x) \in \Gamma}  v(t,x) \right\}.$
		
		Then $w(t,x)$ satisfies conditions,
			\begin{equation} w + w_{t} = |x|^{2-a}\Delta w \label{dh1}\end{equation}
			$$ w|_{\Gamma} \leqq 0 \label{initial4_4}$$
		
		Now we would like to prove $w \leqq 0$.\par
		\vspace{10pt}

		Assume the maximum of w in $Q$, $w(t_{0}, x_{0}) $ is positive. \par
		Then, $(t_{0}, x_{0}) \notin \Gamma$ and
		 $w_{t}(t_{0},x_{0})\geqq 0$. \par
		
		\vspace{5pt}
		Now by \eqref{dh1}, we obtain $\,\Delta w(t_{0},x_{0}) >0 $, but this contradics the assumtion that $w(t_{0}, x_{0}) $ is the maximum.

		\end{proof}

\begin{theorem}[Maximum principle (Unbounded ver.)]\label{maxu}
		Suppose $a>\max\left\{0, 2-N\right\}$. Let $v = v(t,x)$ be a function, satisfying the following conditions: 
		\begin{enumerate}
		
			\item $v(t,x)$ is continuous on $[0,\infty)\times \mathbb{R}^{N}$.
			
			\item $C^{1}$-class about $t$.			
			\item	$C^{2}$-class about $x$ on $\mathbb{R}^{N}\backslash \{0\}$.
		
			\item $a(x) := v(0,x)$ is upper bounded.
			
			\item $\displaystyle v(t,0) \leqq \max_{(t,x) \in \Gamma} a(x)$
			
			\item For any $T >0$, there exist $k >0$, $R >0$ and $C >0$ such that,
			$$v(t,x) \leqq C(1+ |x|^{ka}) \hspace{15pt} \left(t\in[0,T], \hspace{5pt} |x| > R\right)$$

			\item $\displaystyle v_{t}= \frac{1}{a}|x|^{2-a}\Delta v$

		\end{enumerate}
		
		\vspace{5pt}
		Then, 
			$$ \max_{(t,x) \in \mathbb{R}_{\geqq 0}\times \mathbb{R}^{N}} v(t,x) \,\leqq\, \max_{x\in \mathbb{R}^{N}} a(x)$$
	
\end{theorem}

\begin{proof}\quad\par\vspace{3pt}

		By the calculation, $\frac{1}{a}|x|^{2-a}\Delta |x|^{ma}= am(\lambda_{a}+m)|x|^{(m-1)a}$. 
		We are able to construct the solution of $a$-deformed heat equation:
		$$ f_{m}(|x|^{a},t) = |x|^{ma}+c_{1}|x|^{(m-1)a}t+\dots + +c_{m}t^{m}\hspace{25pt}(c_{i} >0 \hspace{10pt} i=0 \dots m).$$
		(In fact, this coincides with $f_{m}(t,x)$ in Corollary \ref{intpoly}. We prove it later in Corollary \ref{polytype}.)

		\vspace{15pt}
		Fix a point $(t_{0},x_{0})\in  [0,T]\times \mathbb{R}^{N}$, and define 
		
		\begin{equation} w(t,x) := v(t,x) - \epsilon f_{k+1}(t,|x|^{a}) -\sup_{x\in \mathbb{R}^{N}}a(x) \hspace{10pt}(\epsilon >0)\end{equation}
		
		For each $\epsilon$, we may choose sufficiently large $T,L>0$ such that $w|_{\Gamma} \leq 0$ and $\quad(t_{0},x_{0}) \in Q$
		(Here we use the notations of Theorem \ref{maxb} for $\Gamma$ and $Q$).\par
		Then, by Theorem \ref{maxb},
		
		\begin{equation}
			v(t_{0}, x_{0}) - \sup_{x\in \mathbb{R}^{N}}a(x) \leqq \epsilon f_{k+1}(t_{0},|x_{0}|^{a})
		\end{equation}
		
		Because $\epsilon$ is arbitrary, we obtain the result.
		
\end{proof}

\subsubsection{Corollaries}

\begin{corollary}[Positivity]\label{positivity}
	Suppose $a>\max\left\{0, 2-N\right\}$. Let $h_{a}(x,y;t)$ be the $a$-deformed heat kernel. Then,
	$$ h_{a}(x,y,t)\geqq 0 $$
\end{corollary}

\begin{proof}\quad\par\vspace{3pt}
	Suppose $u \in C^{\infty}_{\mathrm{cpt}}(\mathbb{R}^{N}\backslash\{0\})$ be a positive function. By Corollaries \ref{3cor}, \ref{Gau} and Proposition \ref{heatflowfou}, we obtain
	$$ H_{t}u(0) \,\leqq\, c_{a}t^{-(\lambda_{a}+1)}\int_{\mathbb{R^N}}e^{-\frac{|y|^{a}}{at} }u(y)\frac{dy}{|y|^{2-a}} \,\leqq\, \mathrm{sup}\, u(x)$$
	and
	\begin{align}
		\MoveEqLeft \left|H_{t}u(x)\right| = \left|\int_{\mathbb{R}^{N}}h_{a}(x,y;t)u(y)\frac{dy}{|y|^{2-a}}\right| =  \left|c_{a}\int_{\mathbb{R}^{N}} \overline{B_{a}(x,\xi)}\,\mathscr{F}_{a}u(\xi)\,e^{-\frac{t}{a}|\xi|^{a}}\frac{d\xi}{|\xi|^{2-a}} \right|\\
		& \leqq {}^\exists C_{0} \left|\int_{\mathbb{R}^{N}} \frac{|B_{a}(x,\xi)|}{\left(1+|\xi|^{a}\right)^{2M+\lceil\frac{N+1}{a}\rceil}}\left|\mathscr{F}_{a}\left(\left(1+|x|^{2-a}\Delta\right)^{2M+\lceil\frac{N+1}{a}\rceil}u(x)\right)(\xi)\right|\frac{d\xi}{|\xi|^{2-a}} \right|\\
		& \leqq {}^\exists C_{1} \left(1+|x|^{M}\right) \left|\int_{\mathbb{R}^{N}} \frac{1}{1+|\xi|^{N+1}}\frac{d\xi}{|\xi|^{2-a}} \right|
		\, \leqq {}^\exists C_{2}\left(1+|x|^{M}\right).
	\end{align}

	Now we may apply Theorem \ref{maxu} to $-H_{t}u(x)$. \par
	Then we obtain $H_{t}u(x)\geqq 0 $. This means $h_{a}(x,y;t)\geqq 0 $

\end{proof}

\begin{corollary}[Preservation of magnitude relation]\label{supinf}
	Let $ u(x), v(x)\in  C\left(\mathbb{R}^{N}\right)$ be polynomial growth functions.
	If $u(x) \leqq v(x)$ for all $x$, then
	$$ H_{t}u(x) \leqq H_{t}v(x)$$
	
	As a special case,
	$$ \inf u \leqq \inf H_{t}u \leqq \sup H_{t}u \leqq \sup u $$
\end{corollary}

\begin{proof}\quad\par\vspace{3pt}
	By,
	$$H_{t}v(x)-H_{t}u(x)=\int_{\mathbb{R}^{N}}h_{a}(x,y;t)\left\{v(y)-u(y)\right\}\frac{dy}{|y|^{2-a}}\geqq 0.$$
\end{proof}

\begin{corollary}[Preservation of polynomial growthness]\label{polygr}
	Let $ u(x) \in C\left(\mathbb{R}^{N}\right)$.
	If $u$ satisfy, $\left|u(x)\right| \leqq C\left(1 + |x|^{ma}\right)$, then
	$$ \left|H_{t}u(x)\right| \leqq C\left(1+f_{m}(x,t) \right)$$
\end{corollary}	
\begin{proof}\quad\par\vspace{3pt}
	Apply Corollary \ref{supinf}, combained with Corollary \ref{intpoly}.
\end{proof}

\begin{corollary}[Initial condition]\label{initial}
	Suppose $u\in C(\mathbb{R}^{N})$ and $u$ is polynomial growth, then $H_{t}u(x)$ satisfies initial value condition, 
	\begin{equation}
		\lim_{t\rightarrow + 0} H_{t}u(x) = u(x)
	\end{equation}
\end{corollary}

\begin{proof}\quad\par\vspace{3pt}
	With reference to the fact that continuous function on $\mathbb{R}^{N}$ is uniformly approximated by smooth function and Corollary \ref{supinf}, we obtain this results.
\end{proof}

\begin{corollary}[Uniqueness theorem]\label{uniq}
	Suppose $u(t,x), v(t,x)$ are solutions of $a$-deformed heat equation satisfying, 1,2,3,6,7 in Theorem \ref{maxu}. If $u(0,x) = v(0,x)$ and $u(t,0)=v(t,0)$, then
	$$ u(t,x) = v(t,x) $$
\end{corollary}

\begin{proof} \quad\par\vspace{3pt}
	Apply Theorem \ref{maxu} to $u-v$.
\end{proof}

\textbf{Remark)}\par
We are also able to prove the maximum principle (Theorem \ref{maxu}), even if we exchange condition 5 in theorem \ref{maxu} into $C^{1}$-ness at $x=0$. 
Then we may remove the condition $u(t,0)=v(t,0)$ in Corollary \ref{uniq}.
If $a > 1 $, we are able to prove the heat flow $H_{t}u(x)$ is in fact $C^{1}$-class at $x=0$, so the assumtion of $C^{1}$-ness is natural in these case. 
On the other hand, when $ 0<a\leqq 1$,  $H_{t}u(x)$ does not $C^{1}$-class. Thus the relationship between the uniqueness theorem and the heat flow operator seems more complicated.

\vspace{7pt}

\begin{corollary}[Composition law (as integral kernel)]\label{comp}
	Suppose $a>\max\left\{0, 2-N\right\}$. Let $h_{a}(x,y;t)$ be the $a$-deformed heat kernel.
	$$  c_{a}\int_{\mathbb{R}^{n}}h_{a}(x,y,t)h_{a}(y,z,s)\frac{dy}{|y|^{2-a}} 
				= h_{a}(x,z,t+s)$$

\end{corollary}

\begin{corollary}[Composition law (as semigroup)]\label{comp2}
	Suppose $a>\max\left\{0, 2-N\right\}$. Let $H_{t}u$ be an $a$-deformed heat flow. 
	$$ H_{t+s}u=H_{t}(H_{s}u)$$
\end{corollary}
\begin{proof}[Proof of \textup{Corollary \ref{comp}} and \textup{Corollary \ref{comp2}}]\quad\par\vspace{3pt}
	To prove Corollary \ref{comp}, we need only to show that $ H_{t+s}u=H_{t}(H_{s}u)$ for any bounded functions $u$. 
	Since, for fixed $s$, both of $H_{t+s}u$ and $H_{t}(H_{s}u)$ are regarded as heat flow with initial state $H_{s}u$, 
	the result follows from the uniqueness theorem (Corollary \ref{uniq}) with reference to Corollary \ref{supinf} and Corollary \ref{inteq}.
	Corollary \ref{comp2} follows from Corollary \ref{comp} immediately.
	\end{proof}

\begin{corollary}[Polynomial type $a$-deformed heat flow]\label{polytype}
	Suppose $a>\max\left\{0, 2-N\right\}$. Let $h_{a}(x,y;t)$ be the $a$-deformed heat kernel and $p \in \mathscr{H}^{l}(\mathbb{R}^{N})$, then
	\begin{align}
		\MoveEqLeft c_{a}\int_{\mathbb{R}^{n}}h_{a}(x,y,t)|y|^{ma}p(y) \frac{dy}{|y|^{2-a}} \\
		& =\left(\sum_{i=0}^{m} a^{i}\binom{m}{i}\frac{\Gamma(\lambda_{a,l}+m+1)}{\Gamma(\lambda_{a,l}+m-i+1)}|x|^{(m-i)a} t^{i}\right) p(x)
	\end{align}

\end{corollary}
\begin{proof}\quad\par\vspace{3pt}

	By the calculation $ \frac{1}{a}|x|^{2-a}\Delta |x|^{ma}p(x) = am\,(\lambda_{a,l}+ m)\,|x|^{(m-1)a}\,p(x)$, 
	we are able to show that, (not only the left hand side, but also) the right hand is a solution of the $a$-generalized heat equation with inital state $|x|^{ma}p(x)$.

	\vspace{5pt}
	Now we would like to apply Corollary \ref{uniq}. We compare the both side at $x=0$.
	When $l>0$, both sides are equal to $0$ (For the calculation of the left hand side, we use Fact \ref{rep} Corollary \ref{Gau}. The right hand side is calculated immediately from its definition). 
	When $l=0$, both sides are equal to $a^{m}\frac{\Gamma(\lambda_{a}+m+1)}{\Gamma(\lambda_{a})}$
	(For the calculation of the left hand side, we use Corollary \ref{intpoly} and recurrence formula \eqref{rec}. The right hand side is calculated immediately from its definition).
	Then, the claim is proved.

\end{proof}

\vspace{10pt}

\section{$a$-deformed Brownian motion}\label{section5}

		In this section, we construct the $a$-deformed Brownian motion.

\begin{definition}[$a$-deformed Brownian motion]\label{defbro}
		Suppose $a>\max\left\{0, 2-N\right\}$. Let $\Omega = (X, \mathfrak{F})$ be a measurable space, 
		$\{P_{x}\}_{x\in\mathbb{R}^{N}}$ be a family of probability measures on $\Omega$, and
		$\left(B_{t}\right)_{t\geqq 0}$ be a $\mathbb{R}^{N}$-valued stochastic process defined on $\Omega$.
		
		\vspace{5pt}
		 If $\left(\{P_{x}\}_{x\in\mathbb{R}^{N}},\left(B_{t}\right)_{t\geqq 0}\right)$ satisfies the following 3 properties, 
		 then we call it as the \mbox{\textbf{$\bf{a}$-deformed Brownian motion}}.
		 
		 \begin{enumerate}
		 	\item $B_{t}(\omega)$ is continuous about $t$ for almost all $\omega \in \Omega$.
			\item $P_{x}\left(\left\{\omega \in \Omega \left.\right| B_{0}(\omega)=x\right\}\right)=1$ 
			\item For $0=t_{0} < t_{1} < \dots <t_{n}$ and $A_{1}, 
				\dots A_{n} \in \mathfrak{B}(\mathbb{R}^{N})$
			\begin{align}
				\MoveEqLeft P_{x}\left(\left\{\omega\in X| B_{t_{i}}(\omega)\in A_{i} \hspace{5pt} \mathrm{for}  
			\hspace{5pt} i = 1, \dots, p\right\}\right) & \\
				& =  \int_{A_{1}} \frac{dx_{1}}{|x_{1}|^{2-a}} 
			\dots \int_{A_{n}}\frac{dx_{n}}{|x_{n}|^{2-a}} \prod_{i=1}^{p} h_{a}(x_{i-1},x_{i}, t_{i}-t_{i-1})
			\end{align}

		Here, $x_{0}=x$ and $\mathfrak{B}(\mathbb{R}^{N})$ is the set of Borel measurable sets 
		with respect to the topology of Euclidean space, and $h_{a}(x,y;t)$ is the $a$-deformed heat kernel.
		 \end{enumerate}
\end{definition}

\begin{lemma}[rephrase of condition\,3 in Definition \ref{defbro}]\label{reph}
	Let $f : \left(\mathbb{R}^{N}\right)^{n} \rightarrow \mathbb{R} $ be a bounded measurable map 
	and $0 =t_{0} < t_{1}<  \dots < t_{n}$ be positive real numbers.
	
	\vspace{5pt}
	Then $\omega \mapsto f(B_{t_{1}}(\omega), \dots ,B_{t_{n}}(\omega))$
	 is  a measurable map on $\Omega$ and 
	 \begin{align}
				\MoveEqLeft \int _{\Omega} f(B_{t_{1}}(\omega), \dots ,B_{t_{n}}(\omega)) dP_{x}(\omega)  & \\
				& =  \int_{A_{1}  \times \dots \times A_{n}} f(x_{1},\dots x_{n})
				\prod_{i=1}^{p}h_{a}(x_{i-1},x_{i}, t_{i}-t_{i-1}) \frac{dx_{1}}{|x_{1}|^{2-a}} \dots \frac{dx_{n}}{|x_{n}|^{2-a}}
			\end{align}

\end{lemma}	

\begin{proof}\quad\par\vspace{3pt}
	We consider the measurable map 
	$\varphi : \Omega \rightarrow (\mathbb{R}^{N})^{n} \hspace{5pt} \varphi(\omega) = (B_{t_{1}}(\omega ), \dots ,B_{t_{n}}(\omega )).$
	Then the map $\omega \mapsto f(B_{t_{1}}(\omega), \dots ,B_{t_{n}}(\omega))$ is equal to $ f\circ\varphi$. This is measurable.
	
	\vspace{10pt}
	By using $\varphi$, we may calculate the left hand side as 
	\begin{align}
		\MoveEqLeft \int _{\Omega} f(B_{t_{1}}(\omega), \dots ,B_{t_{n}}(\omega)) dP_{x}(\omega) =  \int _{\Omega} f\circ\varphi(\omega)dP_{x}(\omega) = \int_{(\mathbb{R}^{N})^{n}}f(x)d(\varphi_{*}P_{x})(x).& 
	\end{align}
	On the other hand, the condition 3 in Definition \ref{defbro} may be rephrased as 
	$$ \varphi_{*}P_{x}(A_{1}\times\dots\times A_{n}) = =  \int_{A_{1}} \frac{dx_{1}}{|x_{1}|^{2-a}} 
	\dots \int_{A_{n}}\frac{dx_{n}}{|x_{n}|^{2-a}} \prod_{i=1}^{p} h_{a}(x_{i-1},x_{i}, t_{i}-t_{i-1}).$$
	Hence, by the Radon-Nikodym theorem, we obtain 
	$$ d(\varphi_{*}P_{x})(x) = \prod_{i=1}^{p} h_{a}(x_{i-1},x_{i}, t_{i}-t_{i-1}) \frac{dx_{1}}{|x_{1}|^{2-a}} \dots \frac{dx_{n}}{|x_{n}|^{2-a}} $$
	and this shows the claim.

\end{proof}

\vspace{5pt}
	
\begin{theorem}[Construction of Brownian motion]\label{Bro}
	Suppose $a>\max\left\{0, 2-N\right\}$. Then there exists a $a$-deformed Brownian motion on a certain measurable space.
\end{theorem}

We need two facts for the proof of Theorem \ref{Bro}.
	
\begin{pfact}[Kolmogolov extension theorem]\label{ext}
		Assume that the family of probability spaces $\Omega_{n} = (\mathbb{R}^{n}, \mathfrak{B}(\mathbb{R}^{n}), \mu_{n}) \hspace{5pt} (n=1,2, \dots)$ 
		has the following compatibility condition:
		
		$$\mu_{n+1}(A \times \mathbb{R})= \mu_{n}(A)$$
		
		then there exists a measure $\mu$ on $(\mathbb{R}^{\infty}, \mathfrak{B}(\mathbb{R}^{\infty}))$ such that 
		
		$$\mu(A \times \mathbb{R^{\infty}})=\mu_{n}(A).$$
	
\end{pfact}
	
Let $\displaystyle D := \left\{ x \in \mathbb{R}\,|\, x = \frac{k}{2^{s}} \hspace{10pt} k \in 2\mathbb{N}+1, s \in \mathbb{N} \right\} = \{t_{1}, t_{2},\dots\}$. 
	
\begin{pfact}[Kolmogolov continuity theorem]\label{conti}
		Let $(S,\alpha)$ be a complete metric space and $(B_{t})_{t\in D}$ be a $S$-valued stochastic process parametrized by $D$.
		
		\vspace{7pt}
		If  there exists $p, \epsilon > 0$ for any $R>0$, such that
		
		$$ \int_{X}\alpha(B_{t}(\omega),B_{s}(\omega))^{p} dP(\omega)\hspace{3pt} \leqq \hspace{3pt} C|t-s|^{1+\epsilon} \hspace{10pt} (\forall t,s\in D\cap[0,R])$$
		
		 then $B_{t}(\omega)$ is continuous about $t$ for almost all $\omega$.
		
		\vspace{5pt}
		In particular, we are able to construct continuous stochastic process $(B_{t})_{t\geqq 0}$ by extending $(B_{t})_{t\in D}$.
		
\end{pfact}
	
\begin{proof}[Proof of \textup{Theorem \ref{Bro}}]
		\quad\par\vspace{3pt}
		
		We fix $x \in\mathbb{R}^{N}$.
		We consider the family of probability spaces 
		$\Omega_{m} = (\mathbb{R}^{m}, \mathfrak{B}(\mathbb{R}^{m}), \mu_{m}) \hspace{5pt} (m=1,2, \dots)$.
		The measures $ \mu_{m} $ defined as follow :
		
		$$ \mu_{m} (A_{1}\times\dots A_{m}) =  \int_{A_{1}} \frac{dx_{1}}{|x_{1}|^{2-a}} 
			\dots \int_{A_{m}}\frac{dx_{p}}{|x_{m}|^{2-a}} \prod_{i=1}^{m} h_{a}(x_{i-1},x_{i}, t_{\sigma(i)}-t_{\sigma(i-1)}) $$
		
		Here, $t_{\sigma(1)} < \dots < t_{\sigma(m)} $ is 
		the permutation of first m elements of $D$, $t_{1},\dots ,t_{m}$, and $x_{0}=x$.
		
		\vspace{3pt}
		The well-definedness and the compatibility conditions of $\mu_{m}$'s are guaranteed by Corollaries \ref{inteq}, \ref{positivity}, \ref{comp}.
		Thus, by Fact \ref{ext}, 
		we obtain a probability space $\Omega_{x} = (\mathbb{R}^{D}, \mathfrak{B}(\mathbb{R}^{D}),\mu_{x})$.
		
		\vspace{10pt}
		We define the stochastic process $(B^{(x)}_{t})_{t\in D}$ by 
		$ B^{(x)}_{t_{i}}(\omega)=\omega_{i} $ and $B^{(x)}_{0}(\omega)=x$.
		
		If we are able to apply Lemma \ref{conti} to this $(B^{(x)}_{t})_{t\in D}$, 
		then we may extend it to $(B^{(x)}_{t})_{t\geqq 0}$, and
		$$\Omega := \coprod_{x\in\mathbb{R_{d}}} \Omega_{x} \hspace{15pt}\text{and}\hspace{15pt}
		\left(\{\mu_{x}\}_{x\in\mathbb{R}^{N}},\left(\coprod B^{(x)}_{t}\right)_{t\geqq 0}\right)$$
		are what we would like to construct.
		So we prove the following lemma.
		
		\begin{plemma}
			Consider the metric $\alpha$ on $\mathbb{R}^{N}$ defined as:
		
			$$ \alpha(r\omega, s\mu) := | r^{ma}\omega - s^{ma} \mu |.$$
		
			\vspace{10pt}
		
			If $m$ is sufficiently large, then
			for any $T>0$, there exists $C>0$ such that
			$$ \int_{\Omega}\alpha(B_{s}(\omega),B_{s+t}(\omega))^{4} dP_{x}(\omega) \hspace{3pt}\leqq\hspace{3pt} Ct^{2}$$
			for any $ t,s \in D\cap [0,T] $.
		\end{plemma}

		\begin{proof}\quad\par\vspace{3pt}
				
			By Theorem \ref{reph},
			\begin{align}
				\MoveEqLeft \int_{X} \alpha(B_{s}(\omega),B_{t+s}(\omega))^{4}d\omega & \\
				& =\int_{\mathbb{R}^2} \alpha(x_{1},x_{2})^{4} h_{a}(x,x_{1},s)h_{a}(x_{1},x_{2};t)\frac{dx_{1}}{|x_{1}|^{2-a}}\frac{dx_{2}}{|x_{2}|^{2-a}} .
			\end{align}
		
			\vspace{5pt}
			Hence, we estimate
			$$ \int_{\mathbb{R}^{N}}\alpha(x_{1},x_{2})^{4}h_{a}(x_{1},x_{2};t)\frac{dx_{2}}{|x_{2}|^{2-a}} .$$
		
			With reference to decomposition like $ \alpha(r\omega,s\nu)^{4}= r^{4ma}+s^{ma}\beta_{0}(r,s,\langle\omega,\nu\rangle)$ (Here, $\beta_{0}(-,-,-)$ is polynimial),
			for sufficient large $m$ and $k=0,1,2$, we are able to justify the equality
			\begin{align}
				\MoveEqLeft \left( \frac{\partial}{\partial t} \right)^{k}\int_{\mathbb{R}^{N}}\alpha(x_{1},x_{2})^{4}h_{a}(x_{1},x_{2};t)\frac{dx_{2}}{|x_{2}|^{2-a}} &\\
		 	& = \int_{\mathbb{R}^{N}}\left\{\left(\frac{1}{a} |x_{2}|^{2-a}\Delta_{x_{2}}\right)^{k}
		 	\alpha(x_{1},x_{2})^{4}\right\}h_{a}(x_{1},x_{2};t)\frac{dx_{2}}{|x_{2}|^{2-a}}. \\
			\end{align}
			
		\vspace{20pt}
		Since we may write 
			\begin{align}
				\MoveEqLeft |x_{2}|^{2-a}\Delta_{x_{2}}\alpha (x_{1},x_{2})^{4} & \\
				& = s^{(m-1)a}\,\beta_{1}(r,s) + \beta_{2}(r,s)\,r^{ma}s^{(m-1)a}\,P_{1}(\omega,\mu)\\
				& \hspace{150pt}+ \beta_{3}(r,s)\,r^{2ma}s^{(2m-1)a}\,P_{2}(\omega,\mu)\\
			\end{align}
			\begin{align}
				\MoveEqLeft \left(|x_{2}|^{2-a}\Delta_{x_{2}}\right)^{2}\alpha (x_{1},x_{2})^{4} & \\
				& = s^{(m-2)a}\,\beta_{4}(r,s) + \beta_{5}(r,s)\,r^{ma}s^{(m-2)a}\,P_{1}(\omega,\mu)\\
				& \hspace{150pt}+ \beta_{6}(r,s)\,r^{2ma}s^{(2m-2)a}\,P_{2}(\omega,\mu)\\
			\end{align}
		 	$$ \text{( Here, } \beta_{1}(-,-), \dots,  \beta_{6}(-,-) \text{ are polynimials. )} $$
		and
		 	$$\alpha(x_{1},x_{1})^{4} =0 ,$$
		 	$$|y|^{2-a}\Delta \alpha(x_{1},x_{1})^{4} =0\,, $$
		 	\vspace{10pt}
		 
		 	for sufficient large $m$, we obtain
		 	$$ \lim_{t\rightarrow 0}\int_{\mathbb{R}^{N}}
		 	\alpha(x_{1},x_{2})^{4}h_{a}(x_{1},x_{2};t)\frac{dx_{2}}{|x_{2}|^{2-a}} = 0$$
		 	$$ \lim_{t\rightarrow 0}\frac{\partial}{\partial t}
		 	\int_{\mathbb{R}^{N}}\alpha(x_{1},x_{2})^{4}h_{a}(x_{1},x_{2};t)\frac{dx_{2}}{|x_{2}|^{2-a}} = 0$$
		 	$$ \left|\frac{\partial^{2}}{\partial t^{2}}
		 	\int_{\mathbb{R}^{N}}\alpha(x_{1},x_{2})^{4}h_{a}(x_{1},x_{2};t)\frac{dx_{2}}{|x_{2}|^{2-a}}\right| \leqq {}^\exists C_{1}(1+ t^{{}^\exists k}+ |x_{1}|^{{}^\exists l})$$
		 	Here, for the first and the second equality, we applied Corollary \ref{initial} and for the third inequality, we use Corollary \ref{polygr}.
		 
		 	\vspace{10pt}
		 	Hence, by Taylor's theorem,
		 	$$ \left|\int_{\mathbb{R}^{N}} \alpha(x_{1},x_{2})^{4}h_{a}(x_{1},x_{2};t)\frac{dx_{2}}{|x_{2}|^{2-a}}\right| \leqq {}^\exists C_{2}\,t^{2}(1+ t^{k}+ |x_{1}|^{l}).$$

			By appling Corollary \ref{polygr} again, we obtain the result.
		 \end{proof}
\end{proof}

\begin{defprop}[standard $a$-deformed Brownian motion]\label{standard}
	Suppose $a>\max\left\{0, 2-N\right\}$. When the $a$-deformed Brownian motion 
	$(\{P_{x}\}_{x\in\mathbb{R}^{N}},\left(B_{t}\right)_{t\geqq 0})$ is defined on the measurable space $(W,\mathfrak{B}(W))$ and $B_{t}(\omega)= \omega_{t}$ 
	(Here, $W := \mathrm{C} \left( [0,\infty),\mathbb{R}^{N}\right)$ and $\mathfrak{B}(W)$ is the set of Borel measurable sets 
	with respect to the compact open topology.), we call it as the \textbf{standard }$\bf{a}$\textbf{-deformed Brownian motion} and we call 
	$\{P_{x}\}_{x\in\mathbb{R}^{N}}$ as the \textbf{$\bf{a}$-deformed Wiener measures}.\par
	\vspace{3pt}
	The standard $a$-deformed Brownian motion uniquely exists.
\end{defprop}

For proof, we need the following a Fact and a Lemma.

\begin{pfact}[$\pi$-$\lambda$ theorem]\label{pilam}
	Let $X$ be a set and $\mathscr{P},\mathscr{D}$ be families of subsets of $X$.
	
	We assume the following.
	\begin{enumerate}
		\item  $ \mathscr{P} $ is $\pi$-system, that is 
			$$ A,B \in \mathscr{P} \hspace{10pt} \Rightarrow \hspace{10pt} A\cap B\in \mathscr{P}$$
		\item $ \mathscr{D}$ is $\lambda$-system, that is
			\begin{enumerate}
				\item $X \in \mathscr{D}$
				\item $A,B\in\mathscr{D},\, A \subset B \hspace{10pt} \Rightarrow \hspace{10pt} B\backslash A \in \mathscr{D}$
				\item $\{A_{n}\} \subset \mathscr{D},\, A_{n}\uparrow A \hspace{10pt} \Rightarrow \hspace{10pt} A \in \mathscr{D}$
			\end{enumerate}
		\item $\mathscr{P} \subset \mathscr{D}$
		
	\end{enumerate}

	\vspace{10pt}
	Then, $$\sigma(\mathscr{P})\subset\mathscr{D}.$$
	Especially, 
	$$\mathscr{P}\subset\mathscr{D}\subset\sigma(\mathscr{P})\hspace{10pt} \Rightarrow \hspace{10pt}\mathscr{D} = \sigma(\mathscr{P}).$$
\end{pfact}

\begin{plemma}\label{P}
	Let $\mathscr{P}$ be a family of all subsets of $W$ such that
	$$\left\{ \,\omega \in W\,; \,\omega_{t_{i}}\in A_{i} \hspace{10pt}(i=1,\dots ,p)\right\}$$ 
	$$(\,\text{Here, } t_{i}\in [0,\infty),\text{ and }A_{i} \text{ is open sets of }\mathbb{R}^{N}\,).$$

	Then $\mathscr{P}$ is $\pi$-system and $\mathfrak{B}(W)=\sigma(\mathscr{P})$.
\end{plemma}

\begin{proof}\quad\par\vspace{3pt}
	We check the latter claim. By definition, $\mathfrak{B}(W)$ is generated by the sets like
	$$ \left\{\omega\in W\,;\,\omega(K)\subset Z\right\}$$
	$$(\text{ Here, } K \text{ is compact set of } [0,\infty) \text{ and } Z \text{ is closed set of }\mathbb{R}^{N}\,).$$
	Let $\{t_{m}\}_{m\in \mathbb{N}}$ be a countable dense subset of $K$. Then,
	$$\bigcap_{m\in{\mathbb{N}}} \,\left\{\,\omega_{t_{m}}\in Z\,\right\}=\left\{\,\omega\left(\{t_{m}\}_{m\in\mathbb{N}}\right)\subset Z\,\right\} = \left\{\,\omega(K)\subset Z\right\},$$
	so we obtain the result.
\end{proof}

\vspace{10pt}
\begin{proof}[proof of \textup{Theorem \ref{standard}}]\quad\par\vspace{3pt}
	We only need to show existence and uniqueness of the $a$-deformed Wiener measures.
	\begin{enumerate}
		\item (existence)\par
		Let $(\{P'_{x}\}_{x\in\mathbb{R}^{N}},\left(B'_{t}\right)_{t\geqq 0})$ be an $a$-deformed Brownian motion on $\Omega = (X,\mathfrak{F})$ constructed in Theorem \ref{Bro}. 
		We consider the pushout by the map: 
		$$\varphi : X \longrightarrow C\left([0,\infty),\mathbb{R}^{N}\right) \hspace{20pt} \omega \longmapsto  \left(t\mapsto B_{t}(\omega)\right)  .$$
	
		Then, we obtain the probability measures $\{\varphi_{*}P_{x}\}_{x\in\mathbb{R}^{N}}$ and these are $a$-deformed Wiener measures.

		\item (uniqueness)\par
			Let $\{P_{x}\}_{x\in\mathbb{R}^{N}}$ and $\{P'_{x}\}_{x\in\mathbb{R}^{N}}$ be $a$-deformed Wiener measure.
			We consider the set
			$$ \mathscr{D}:=\left\{\,F \in \mathfrak{B}(W) \,;\, P_{x}(F)=P_{x}'(F) \hspace{10pt}\forall x\in\mathbb{R}^{N} \,\right\}.$$
			Then, $\mathscr{D}$ is $\lambda$-system contain $\mathscr{P}$ in Fact \ref{P}. 
			Hence, with reference to Fact \ref{pilam}, we obtain the result.

	\end{enumerate}

\end{proof}

\begin{definition}[Basic concepts]
	Let 
	$\left(\{P_{x}\}_{x\in\mathbb{R}^{N}},\left(B_{t}\right)_{t\geqq 0}\right)$ be the standard $a$-deformed Brownian motion. 
	We define the following concepts;
	
	\begin{enumerate}
		\item (Filtration)
		\begin{align}
			\MoveEqLeft \mathfrak{B}(W)_{t} := \sigma(B_{s};s\leqq t) &\\
			& \text{=( minimal }\sigma\text{-algebra such that }\{B_{s};s\leqq t\}\text{ are measurable )}
		\end{align}

		\item (Expected value)\par
			Let $f$ be a bounded measurable function. The expected values of it are defined as
			$$ E_{x}(f) := \int_{W} f(\omega)dP_{x}(\omega). $$
		
		\item (Conditional expected value) \par
		Let $f$ be a bounded measurable function and $\mathscr{G}$ be a sub$\,\sigma$-algebra of $\mathfrak{B}(W)$. 
		The conditional expected value of $f$ with respect to $\mathscr{G}$ is the $\mathscr{G}$ measurable bounded function $g$ satisfying the following condition;

		$$E_{x}(f1_{G})=E_{x}(g1_{G}) \hspace{20pt}(\, \forall G \in \mathscr{G}\, ).$$

		The Radon-Nikodym theorem shows that the conditional expectation value exsits uniquely except the  differ on the set of measure zero. 
		We write it as $g = E_{x}(f\left.\right|\mathscr{G})$.
		
		\item (Time shift operator)\par
		We define time shift operator $\theta_{t}: W \rightarrow W$ by $B_{t}(\theta_{s}\omega)= B_{t+s}(\omega)$.
	\end{enumerate}
	
\end{definition}

Now, we prove some Markovness of the $a$-deformed Brownian motion. In proof, we apply the following technique of measure theory.
\begin{pfact}[Monotone class theorem]\label{mono}
	Let $X$ be a set, $\mathscr{F}$ be a finite additive class of $X$ and $\mathscr{M}$ be a monotone class of $X$. Then
	$$\mathscr{F} \subset \mathscr{M} \hspace{5pt}\Rightarrow\hspace{5pt} \sigma(\mathscr{F})\subset \mathscr{M}.$$
	Especially, 
	$$\mathscr{F} \subset \mathscr{M} \subset \mathscr{\sigma(\mathscr{F})}\hspace{5pt}\Rightarrow\hspace{5pt} \mathscr{M} =\sigma(\mathscr{F}).$$
	
\end{pfact}

\begin{plemma}\label{ftdush}
	Let $\left(\{P_{x}\}_{x\in\mathbb{R}^{N}},\left(B_{t}\right)_{t\geqq 0}\right)$ be the standard $a$-deformed Brownian motion. 
	Let $\mathfrak{B}(W)_{t} '$ be the family of all measurable sets of $W$ written as
	$$ \left(B_{t_{1}}(\omega),\dots \,B_{t_{n}} (\omega)\right)^{-1}(A)  \hspace{20pt}(\,A\in \mathfrak{B}\left(\mathbb{R}^{N\times n}\right),\hspace{5pt} 0=t_{0}\leqq t_{1}\leqq\dots\leqq t_{n}= t\,).$$
	Then $\mathfrak{B}(W)_{t} '$ is finite additive class and $\sigma(\mathfrak{B}(W)_{t}')= \mathfrak{B}(W)_{t}$.
	
\end{plemma}

\begin{proof}\quad\par\vspace{3pt}
	The claim immediately follows from definitions.
\end{proof}

\begin{proposition}[Markov property]\label{markov1}
	Let $\left(\{P_{x}\}_{x\in\mathbb{R}^{N}},\left(B_{t}\right)_{t\geqq 0}\right)$ be the standard $a$-deformed Brownian motion and $f$ be a $\mathfrak{B}(W)_{s}$-measurable bounded function.
	Then, $$E_{x}\left(f(\theta_{t}\omega)\left.\right|\mathfrak{B}(W)_{t}\right) = E_{B_{t}}(f)$$

\end{proposition}

\begin{proof}\quad\par\vspace{3pt}
	By monotone convergence theorem, we may reduce the claim to the case of
	$ f = 1_{A}(B_{s})$ (Here, $A\in\mathfrak{B}(\mathbb{R}^{N})$ and $s \in [0,\infty)$). That is, we only need to prove the equality
	$$ E_{x}(1_{A}(B_{t+s})\left|\mathfrak{B}(W)_{t}\right.)(\omega) = E_{B_{t}(\omega)}(1_{A}(B_{s})) $$
	
			Since the family of sets 
			$$ \mathscr{M}:=\left\{F\in \mathfrak{B}(W)_{t}\,;\, E_{x}\left(1_{A}(B_{t+s})1_{F}\right)= E_{x}\left(E_{B_{t}}(1_{A}(B_{s})1_{F})\right)\right\} $$ is a monotone class of $\mathfrak{B}(W)_{t}$,
			we need to show that $\mathscr{M}$ contain $ \mathfrak{B}(W)_{t}'$ in Lemma \ref{ftdush}, with reference to Fact \ref{mono}.

			The calculation
			\begin{align}
				\MoveEqLeft E_{x}\left(1_{A}(B_{t+s})1_{A'}\left(B_{t_{1}},\dots B_{t_{n}}\right)\right) & \\
				&= \int_{\mathbb{R}^{N}}1_{A'}(y_{1}, \dots y_{n})
				f(y_{n+1})\prod_{i=1}^{n} h_{a}(y_{i-1},y_{i}\,;\, t_{i}-t_{i-1})h_{a}(y_{n},y_{n+1}, s) \\
				& =E_{x}(E_{B_{t}}(1_{A}(B_{s}))1_{A'}(B_{t_{1}},\dots B_{t_{n}}))
			\end{align}

		shows the claim.

\end{proof}

\vspace{10pt}
For the bounded continuous function $V \in \mathrm{C}_{b}(\mathbb{R}^{N})$, we define
$$ A_{t}(\omega) := \exp\left(-\int_{0}^{t} V(B_{s}(\omega))ds \right)$$
Then $A_{t+s}(\omega)=A_{t}(\omega)A_{s}(\theta_{t}\omega)$.

\begin{defprop}
	For a bounded measurable function $f$, we define
	$$ T^{V}_{t}f(x) := E_{x}\left(A_{t}f(B_{t})\right).$$
	Then $T^{V}_{t}f(x)$ is a bounded measurable function and 
	$$  T^{V}_{t+s}f(x)=T^{V}_{t} \left(T^{V}_{s}f(x)\right).$$
\end{defprop}

\begin{proof}\quad\par\vspace{3pt}
	The boundedness of $ T^{V}_{t}f(x)$ follows from the boundedness of the integrand. 
	Proof of the measurability of $ T^{V}_{t}f(x)$ reduces to the measurability of $E_{x}(f(B_{t_{1}},\dots B_{t_{p}}))$ which is continuous.
	The compsition law is showed as follow:
	\begin{align}
		\MoveEqLeft T^{V}_{t+s}f(x)=E_{x}\left(A_{t+s}(\omega)f(B_{t+s}(\omega))\right) & \\
		& = E_{x}\left(A_{t}(\omega)A_{s}(\theta_{t}\omega)f(B_{s}(\theta_{t}\omega))\right) \\
		& = E_{x}\left(A_{t}(\omega)E_{x}\left(A_{s}(\theta_{t}\omega)f(B_{s}(\theta_{t}\omega))\left.\right| \mathfrak{F}_{t}\right)\right) \\
		& = E_{x}\left(A_{t}(\omega)E_{B_{t}}\left(A_{s}(\omega)f(B_{s}(\omega))\right)\right) \\
		& = T^{V}_{t} T^{V}_{s}f(x)\,.
	\end{align}
	At the fourth equality, we applied Proposition \ref{markov1}.
\end{proof}

\begin{lemma}
	The following properties hold:
	\begin{enumerate}
		\item For a continuous bounded function $f\in\mathrm{C}_{b}(\mathbb{R}^{N})$,
			$$\underset{t\rightarrow +0}{\lim}  T^{V}_{t} f(x) = f(x)$$
		\item For $f\in\mathrm{C}^{2}(\mathbb{R}^{N}\backslash\{0\})\cap\mathrm{C}_{b}(\mathbb{R}^{N})$, and if $|x|^{2-a}\Delta f(x)$ is bounded, then
			$$ \frac{\partial}{\partial t} T^{V}_{t} f(x) = T^{V}_{t}(\left\{ \frac{1}{a}|x|^{2-a}\Delta - V(x)\right\}f(x))$$
	\end{enumerate}
\end{lemma}

\begin{proof}\quad\par\vspace{3pt}
	\begin{enumerate}
		\item 
		\begin{align}
			\MoveEqLeft \underset{t\rightarrow +0}{\lim}  T^{V}_{t} f(x) = \underset{t\rightarrow +0}{\lim} E_{x}(A_{t}f(B_{t})) & \\
			& = E_{x}( \underset{t\rightarrow +0}{\lim} A_{t}f(B_{t}))  \\
			& = E_{x}(f(B_{0})) = E_{x}(f(x)) \\
			& = f(x)
		\end{align}
		\item At first, we prove 
		$$\left. \frac{\partial}{\partial t}\right|_{t=+0} T^{V}_{t} f(x) =\left\{ \frac{1}{a}|x|^{2-a}\Delta - V(x)\right\}f(x)$$

		\begin{align}
			\MoveEqLeft \left. \frac{\partial}{\partial t}\right|_{t=+0} T^{V}_{t} f(x)  &\\
			& = \underset{\epsilon\rightarrow +0}{\lim}\frac{T^{V}_{\epsilon}-1}{\epsilon}f(x) \\
			& =  \underset{\epsilon\rightarrow +0}{\lim} \left(E_{x}\left(\frac{A_{\epsilon}-1}{\epsilon}f(B_{\epsilon})\right)+\frac{H_{\epsilon}f(x)-f(x)}{\epsilon}\right) \\
			& = -V(x)f(x)+ \frac{1}{a} |x|^{2-a}\Delta f(x).
		\end{align}

		Then, we are able to calculate the left-hand and right-hand derivatives as follow:
		\begin{align}
			\MoveEqLeft \underset{\epsilon\rightarrow +0}{\lim}\frac{T^{V}_{t+\epsilon}-T^{V}_{t}}{\epsilon}f(x) = \underset{\epsilon\rightarrow +0}{\lim}T^{V}_{t}\left(\frac{T^{V}_{\epsilon}-1}{\epsilon}f(x)\right)(x) &\\
			& = \underset{\epsilon\rightarrow +0}{\lim}E_{x}\left(A_{t}\frac{T^{V}_{\epsilon}-1}{\epsilon}f(B_{t})\right) = E_{x}\left( \underset{\epsilon\rightarrow +0}{\lim} A_{t} \frac{T^{V}_{\epsilon}-1}{\epsilon}f(B_{t})\right)\\
			&=  T^{V}_{t}\left( \frac{1}{a}|x|^{2-a}\Delta - V(x)f(x)\right)\\
		\end{align}

		\begin{align}
			\MoveEqLeft \underset{\epsilon\rightarrow +0}{\lim}\frac{T^{V}_{t}-T^{V}_{t-\epsilon}}{\epsilon}f(x) = \underset{\epsilon\rightarrow +0}{\lim}T^{V}_{t-\epsilon}\left(\frac{T^{V}_{\epsilon}-1}{\epsilon}f(x)\right)(x) &\\
			& = \underset{\epsilon\rightarrow +0}{\lim}E_{x}\left(A_{t-\epsilon}\frac{T^{V}_{\epsilon}-1}{\epsilon}f(B_{t-\epsilon})\right) = E_{x}\left( \underset{\epsilon\rightarrow +0}{\lim} A_{t-\epsilon} \frac{T^{V}_{\epsilon}-1}{\epsilon}f(B_{t-\epsilon})\right)\\
			&=  T^{V}_{t}\left( \frac{1}{a}|x|^{2-a}\Delta - V(x)f(x)\right)\\
		\end{align}
	
	\end{enumerate}
\end{proof}

\begin{theorem}[Feynman-Kac type formula]\label{fey}
	Suppose $a>\max\left\{0, 2-N\right\}$. Let $u(t,x)\in \mathrm{C}_{b}([0,\infty)\times\mathbb{R}^{N})$ be the bounded continuous function satisfying the following conditions:
	\begin{enumerate}
		\item $C^{1}$-class about $t$.
		\item $C^{2}$-class about $x$ when $x\neq 0 $.
		\item $|x|^{2-a}\Delta u(t,x)$ is bounded.
		\item $u_{t}= \left(\frac{1}{a} |x|^{2-a}\Delta-V(x)\right) u$.
	\end{enumerate}

	Then, 
	$$u(t,x)=T_{t}^{V}f(x)=E_{x}\left(e^{-\int_{0}^{t}V(B_{s})ds}f(B_{t})\right)$$
	Here, $f(x):=u(0,x)$.
\end{theorem}

\begin{proof}\quad\par\vspace{3pt}

	Consider the function $f(t,x) \in \mathrm{C}_{b}([0,\infty)\times\mathbb{R}^{N})$ such that $C^{1}$-class about $t$, $C^{2}$-class about $x$ at $x\neq 0$, and $\partial_{t}f(t,x)$ and $|x|^{2-a}\Delta f(t,x)$ are bounded. Then, 

	$$\frac{\partial}{\partial t} T^{V}_{t}(f(t,-))(x) = T^{V}_{t}\left(\left(\frac{1}{a} |x|^{2-a}\Delta-V\right)f(t,-)+\frac{\partial}{\partial t}f(t,-) \right)$$

	By integrating about $t$,

	$$ T^{V}_{t_{0}}(f(t_{0},-))(x) = f(0,x) + \int_{0}^{t_{0}} T^{V}_{s}\left(\left(\frac{1}{a} |x|^{2-a}\Delta-V\right)f(t_{0},-)+\frac{\partial}{\partial t}f(t_{0},-) \right)ds. $$

	By applying $f(t_{0},x)=u(t_{0}-t,x)$, we obtain
	$$u(t, x)= T^{V}_{t}f(x).$$
\end{proof}

\section{Appendix (about special functions)} \label{App}
	In this section, we review the definitions and the properties of special functions. 
Please use as a reference, if you like.

\subsection{Laguerre polynomials}

We review some facts about Laguerre polynimials. 
This part is based on the \cite[Section 8]{MR2401813} and \cite[Section 3.3]{MR2956043} with the same notations.

  \begin{definition}[{\cite[Section 8.1]{MR2401813}}]

    We define the Laguerre polynomials by
    \begin{align}
      \MoveEqLeft L^{(\lambda)}_{l}(x):= \frac{x^{-\lambda}e^{x}}{l!}\frac{d^{l}}{dx^{l}}(x^{\lambda+l}e^{-x}) & \\
      & =\frac{(\lambda+1)_{l}}{l!}\sum_{j=0}^{l} \frac{(-l)_{j}}{(\lambda+1)_{j}}\frac{x^{j}}{j!}\\
    \end{align}

  \end{definition}

  \begin{fact}[{\cite[Section 8.1]{MR2401813}}]

    The Laguerre polynomial $L^{(\lambda)}_{l}(x)$ solves the following equation.
    $$ xu'' +(\lambda +1 -x)u' + lu =0$$

  \end{fact}

  \begin{fact}[{\cite[Section 8.1]{MR2401813}}]
    If $\lambda > -1$, the Laguerre polynomials $\left\{ L^{(\lambda)}_{l}(x) : l=0,1,\dots \right\}$ are complete in $L^{2}(\mathbb{R}_{>0},x^{\lambda}e^{-x})$ and satisfy the orthogonality relation.
    $$\int_{0}^{\infty}L_{m}^{(\lambda)}(x)L_{n}^{(\lambda)}(x) x^{\lambda}e^{-x}dx = \frac{\Gamma(\alpha + n + 1)}{n!}\delta_{mn}$$
  \end{fact}

%

\subsection{Gamma Function}

  We review some integral representations of the Gamma function. 
  These are going to be used when we derive the integral representation of Bessel functions. 
  This part is based on \cite[Chapter 12]{MR1424469}.
  
  \begin{fact}[{\cite[Section 12.22]{MR1424469}}]\label{intgamma1}
    
    A Gamma Function has the following integral representation:
    $$ \Gamma(z) = - \frac{1}{2i \,\mathrm{sin}\pi z}\int_{C_{R}} e^{-t}(-t)^{z-1}dt $$
    Here, $-t=e^{-\pi i}t$ and the integral path $C_{R}$ is the sum of these three paths:
    
    \begin{enumerate}
      \item[($C_{1}$)] Half straight line:  $\left(+\infty\right) e^{2\pi i } \rightarrow R e^{2\pi i}$
      \item[($C_{2}$)] Counterclockwise circle: $Re^{i 0} \rightarrow R e^{2\pi i}$
      \item[($C_{3}$)] Half straight line:  $R e^{ i 0} \rightarrow \left(+\infty\right)e^{ i 0}$
    \end{enumerate}

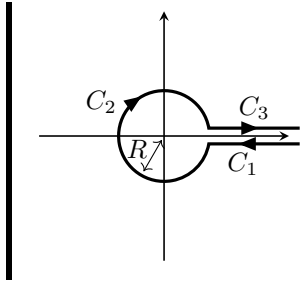
\begin{tikzpicture}
  \draw[->,>=stealth,semithick] (-1.65,0)--(1.65,0)node[above]{}; 
  \draw[->,>=stealth,semithick] (0,-1.65)--(0,1.65)node[right]{}; 
  \draw (0,0)node[below  left]{}; 
  \draw[very thick](1.79088,0.10419)--(10:0.6)arc(10:350:0.6)--(1.79088,-0.10419);

  \coordinate (E) at (1.19088,0.10419); 
  \coordinate (F) at (1.04088,0.16419); 
  \coordinate (G) at (1.04088,0.04419); 
  \fill[black] (E)--(F)--(G)--cycle; 
  \draw[very  thick] (E)--(F)--(G)--cycle; 
\draw (1.19088,0.10419)node[above]{$C_{3}$}; 

\coordinate (E2) at (1.04088,-0.10419); 
  \coordinate (F2) at (1.19088,-0.16419); 
  \coordinate (G2) at (1.19088,-0.04419); 
  \fill[black] (E2)--(F2)--(G2)--cycle; 
  \draw[very  thick] (E2)--(F2)--(G2)--cycle; 
\draw (1.04088,-0.10419)node[below]{$C_{1}$}; 

  \coordinate (E3) at (-0.34415,0.49149); 
  \coordinate (F3) at (-0.50143,0.4546); 
  \coordinate (G3) at (-0.4326,0.35631); 
  \fill[black] (E3)--(F3)--(G3)--cycle; 
  \draw[very  thick] (E3)--(F3)--(G3)--cycle; 
\draw (-0.50143,0.4546)node[above,left]{$C_{2}$}; 

  \draw[thin,<->] (240:0.06) -- (240:0.54);
\draw (240:0.2)node[left]{$R$};
\end{tikzpicture}

  \end{fact}

  \begin{proof}[sketch of proof]\quad\par\vspace{3pt}
    
    When $\mathrm{Re}(z) >0 $,
    \begin{align}
      \MoveEqLeft \int_{C_{R}} e^{-t}(-t)^{z-1}dt = \lim_{R\rightarrow +0} \int_{C_{R}} e^{-t}(-t)^{z-1}dt &\\
      & = -e^{i \pi z}\int_{0}^{\infty} e^{-x}x^{z-1}dx + e^{-i \pi z}\int_{0}^{\infty} e^{-x}x^{z-1}dx \\
      & = -2i\,\mathrm{sin}(z) \Gamma(z)
    \end{align}

    By analytic continuation, we obtain the formula.
  \end{proof}

  \begin{theorem}[{\cite[Section 12.22]{MR1424469}} ]\label{intgamma2}
    
    The reciprocal Gamma Function has the following integral representation:
    $$ \frac{1}{\Gamma(z)} = \frac{i}{2\pi}\int_{C_{R}} e^{-t}(-t)^{-z}dt $$
    Here, integral path $C_{R}$ is the same one as in Theorem \ref{intgamma1}.
  \end{theorem}

  \begin{proof}\quad\par\vspace{3pt}
    The result follows from Theorem \ref{intgamma1} and Eular's reflection formula
    $$\Gamma(z)\Gamma(1-z)= \frac{\pi}{\mathrm{sin}(\pi z)} \,.$$
  \end{proof}

%

\subsection{Bessel functions}
We review some facts of Bessel functions. 
The notations in this section is the same as \cite{MR2401813} and \cite{MR2956043}.

  \begin{definition}
    
    We define the following 4 variations of Bessel functions. 
    \begin{enumerate}
      
      \item Bessel function
      $$ J_{\nu}(w) := \left(\frac{w}{2}\right)^{\nu}\sum_{m=0}^{\infty}\frac{(-1)^{m}(w/2)^{2m}}{m!\Gamma(\nu+m+1)}$$
      
      \item I-Bessel function
      $$ I_{\nu}(w) := \left(\frac{w}{2}\right)^{\nu}\sum_{m=0}^{\infty}\frac{(w/2)^{2m}}{m!\Gamma(\nu+m+1)}$$
      
      \item Normalized Bessel function {\cite[Section 8.5]{MR2401813}} 
      $$ \tilde{J}_{\nu}(w) := \sum_{m=0}^{\infty}\frac{(-1)^{m}(w/2)^{2m}}{m!\Gamma(\nu+m+1)}$$

      \item Normalized I-Bessel function {\cite[Section 8.5]{MR2401813}}
      $$ \tilde{I}_{\nu}(w) := \sum_{m=0}^{\infty}\frac{(w/2)^{2m}}{m!\Gamma(\nu+m+1)}$$

    \end{enumerate}

  \end{definition}

  From the definition, we may obtain,

  \begin{theorem}\label{besseldiff}

    $$\frac{\partial}{\partial w}\tilde{I}_{\nu}(w) =\frac{w}{2}\tilde{I}_{\nu + 1}(w)$$
  
  \end{theorem}

  \vspace{10pt}
  We refer to the integral representation of Normalized I-Bessel function. 
  Theorem \ref{lem bessel} is the change-of-variable version of this formula.
  As reference, see {\cite[Section 12.22]{MR1424469}}.

  \begin{theorem}\label{intbessel1}

    $$ \tilde{I}_{\nu}(w) = \frac{1}{2\pi i} \int_{C_{R}}(-t)^{\nu}e^{-t-\frac{w^{2}}{4t}} \frac{dt}{t} $$

    Here, $-t=e^{-\pi i}t$ and the integral path $C_{R}$ is the sum of these three paths.
    \begin{enumerate}
      \item[($C_{1}$)] Half straight line:  $(+\infty) e^{2\pi i} \rightarrow  R e^{2\pi i}$
      \item[($C_{2}$)] Counterclockwise circle: $Re^{i 0} \rightarrow R e^{2\pi i}$
      \item[($C_{3}$)] Half straight line:  $R e^{i 0} \rightarrow (+\infty) e^{i0}$
    \end{enumerate}
    (This integral path is the same one as in Theorem \ref{intgamma2})

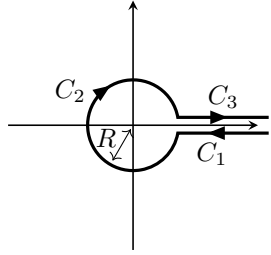
\begin{tikzpicture}
  \draw[->,>=stealth,semithick] (-1.65,0)--(1.65,0)node[above]{}; 
  \draw[->,>=stealth,semithick] (0,-1.65)--(0,1.65)node[right]{}; 
  \draw (0,0)node[below  left]{}; 
  \draw[very thick](1.79088,0.10419)--(10:0.6)arc(10:350:0.6)--(1.79088,-0.10419);

  \coordinate (E) at (1.19088,0.10419); 
  \coordinate (F) at (1.04088,0.16419); 
  \coordinate (G) at (1.04088,0.04419); 
  \fill[black] (E)--(F)--(G)--cycle; 
  \draw[very  thick] (E)--(F)--(G)--cycle; 
\draw (1.19088,0.10419)node[above]{$C_{3}$}; 

\coordinate (E2) at (1.04088,-0.10419); 
  \coordinate (F2) at (1.19088,-0.16419); 
  \coordinate (G2) at (1.19088,-0.04419); 
  \fill[black] (E2)--(F2)--(G2)--cycle; 
  \draw[very  thick] (E2)--(F2)--(G2)--cycle; 
\draw (1.04088,-0.10419)node[below]{$C_{1}$}; 

  \coordinate (E3) at (-0.34415,0.49149); 
  \coordinate (F3) at (-0.50143,0.4546); 
  \coordinate (G3) at (-0.4326,0.35631); 
  \fill[black] (E3)--(F3)--(G3)--cycle; 
  \draw[very  thick] (E3)--(F3)--(G3)--cycle; 
\draw (-0.50143,0.4546)node[above,left]{$C_{2}$}; 

  \draw[thin,<->] (240:0.06) -- (240:0.54);
\draw (240:0.2)node[left]{$R$};
\end{tikzpicture}

  \end{theorem}

  \begin{proof}[sketch of proof]\quad\par\vspace{3pt}
    
    \begin{align}
      \MoveEqLeft \tilde{I}_{\nu}(w) = \sum_{m=0}^{\infty}\frac{(w/2)^{2m}}{m!\Gamma(\nu+m+1)} &\\
      & = \sum_{m=0}^{\infty}\frac{i}{2\pi}\int_{C_{R}}e^{-t}(-t)^{-(\nu + m + 1)}\frac{\left(w/2\right)^{2m}}{m!} dt \\
      & = \frac{1}{2\pi i}\int_{C_{R}}(-t)^{\nu}e^{-t} \sum_{m=0}^{\infty} \frac{(-w^{2}/4t)^{m}}{m!} \frac{dt}{t} \\
      & = \frac{1}{2\pi i} \int_{C_{R}}(-t)^{\nu}e^{-t-\frac{w^{2}}{4t}} \frac{dt}{t}
    \end{align}

    For the second equality, we use Theorem \ref{intgamma2}.
  \end{proof}

  \vspace{10pt}
  We also refer to the another integral representation.

  \begin{theorem}[\cite{MR2401813} Section 8.5]\label{intbessel2}

    If $\nu > -\frac{1}{2} $ then, 
    $$ \tilde{I}_{\nu}(w) =  \frac{1}{\Gamma(\nu+\frac{1}{2})\Gamma(\frac{1}{2})} 
    \int_{-1}^{1}e^{wt}(1-t^2)^{\nu-1/2}dt$$
  \end{theorem}

  \begin{proof}[sketch of proof]\quad\par\vspace{3pt}
    \begin{align}
      \MoveEqLeft \tilde{I}_{\nu}(w) := \sum_{m=0}^{\infty}\frac{(w/2)^{2m}}{m!\Gamma(\nu+m+1)} & \\
      & =  \sum_{m=0}^{\infty} \frac{(w/2)^{2m}}{m!}\frac{1}
      {\Gamma(m+\frac{1}{2})\Gamma(\nu + \frac{1}{2})}\int _{0}^{1}x^{m-1/2}(1-x)^{\nu -1/2} dx \\
      & =\frac{1}{\Gamma(\nu+\frac{1}{2})\Gamma(\frac{1}{2})} 
      \sum_{m=0}^{\infty} \frac{w^{2m}}{(2m)!}\int_{-1}^{1}t^{2m}(1-t^2)^{\nu-1/2}dt \\
      & =\frac{1}{\Gamma(\nu+\frac{1}{2})\Gamma(\frac{1}{2})} \int_{-1}^{1}e^{wt}(1-t^2)^{\nu-1/2}dt \\
    \end{align}
  \end{proof}

  From this, we are able to estimate Bessel functions.
  \begin{corollary}[{\cite[Section 8.5]{MR2401813}} ] \label{ineqbessel}

    If $\nu > -\frac{1}{2} $, then 
    $$|\tilde{I}_{\nu}(w)| \leqq \frac{e^{|\mathrm{Re}(w)|}}{\Gamma(\nu +1)}$$

  \end{corollary}

\vspace{10pt}
\begin{proposition}\label{esb2}
  Suppose $-1<\nu\leqq-\frac{1}{2}$. \par
  There exist constants $A,B,C(\,=A_{\nu},B_{\nu},C_{\nu})$ such that:
  $$ \left|\tilde{I}_{\nu}(w)\right| < (\,A+B\,|w|^{2}+c\,|w|^{4}\,)\,e^{|w|}$$
  
\end{proposition}
\begin{proof}\quad\par\vspace{3pt}
  By the calculation
  \begin{align}
    \MoveEqLeft \left|\tilde{I}_{\nu}(w)\right| \leqq \sum_{m=0}^{\infty} \frac{|w/2|^{2m}}{m!\Gamma(\nu+m+1)} \\
    & = \sum_{m=0}^{2} \frac{|w/2|^{2m}}{m!\Gamma(\nu+m+1)}  + \left|\frac{w}{2}\right|^{6}\sum_{m=0}^{\infty} \frac{|w/2|^{2m}}{(m+3)!\Gamma(\nu+m+4)} \\
    & \leqq \sum_{m=0}^{2} \frac{|w/2|^{2m}}{m!\Gamma(\nu+m+1)}  + \left|\frac{w}{2}\right|^{2}\sum_{m=0}^{\infty} \frac{|w/2|^{2m+4}}{(m+2)!\Gamma(m+5/2)}\\
    & \leqq \sum_{m=0}^{2} \frac{|w/2|^{2m}}{m!\Gamma(\nu+m+1)}+ \frac{|w|^{2}}{4\pi^{1/2}}\sum_{m=0}^{\infty} \frac{|w|^{2m}}{(2m)!} \\
    & \leqq  \sum_{m=0}^{2} \frac{|w/2|^{2m}}{m!\Gamma(\nu+m+1)} + \frac{|w|^{2}}{4\pi^{1/2}} \mathrm{cosh}|w|, \\
  \end{align}
  we obtain the result.
  At the third equality, we applied the following fact: \par
  \begin{plemma}
    If $2\leqq x \leqq y$, then $\Gamma(x) \,\leqq\, \Gamma(y)$.
  \end{plemma}
  \begin{proof}[sketch of proof]\quad\par\vspace{3pt}
    Combining the facts
    $$\Gamma''(x)= \int_{0}^{\infty}t^{x-1}e^{-t}(\mathrm{log}t)^{2}dt \geqq 0$$
    $$\Gamma'(2)= 1-\gamma \geqq 0,$$
    we obtain the result. Here, $\gamma$ is Eular's constant.
  \end{proof}
  
\end{proof}

\subsection{Spherical harmonics}
  We review some facts about spherical harmonics. As reference, see \cite[Chapter 1, Chapter 2]{MR3060033}.

  \vspace{10pt}
  Let $\mathscr{P}(\mathbb{R}^{N})$ be the space of polynomials and $\mathscr{P}^{m}(\mathbb{R}^{N})$ be the space of polynomials with degree $m$. 
  $\mathscr{H}^{m}(\mathbb{R}^{N})$ be the space of harmonic polynomials with degree $m$.

  \begin{plemma}\label{polrep}

    $$ \mathscr{P}^{m}(\mathbb{R}^{N}) = \bigoplus_{i=0}^{[m/2]}||x||^{2i}\mathscr{H}^{m-2i}(\mathbb{R}^{N}) $$
  
  \end{plemma}

  \begin{proof}[sketch of proof]\quad\par\vspace{3pt}
    
    For $p(x) \in \mathscr{H}^{m-2i}(\mathbb{R}^{N})$,
    $$\Delta ||x||^{2i}p(x) =2i(N+ m-1) ||x||^{2i-2} p(x).$$
    From this, we obtain the assertion inductively.

  \end{proof}

  \begin{pdef}
    We define a Hermitian inner product $\langle -,-\rangle_{S^{N-1}}$ on $C(S^{N-1})$ by
    $$ \langle f,g\rangle_{S^{N-1}} := \frac{1}{\mathrm{vol}(S^{N-1})}\int_{S^{N-1}}f(\omega)\overline{g(\omega)}d\omega .$$
    Here, $d\omega$ is $SO(N)$ invariant measure on $S^{N-1}$. \par
    We write the completion of $C(S^{N-1})$ with $\langle -,-\rangle_{S^{N-1}}$ as $L^{2}(S^{N-1})$.
  \end{pdef}

  \begin{pthm}\label{eigen}
    For $p \in\mathscr{H}^{m}(\mathbb{R}^{N})$,
    $$ \Delta_{S^{N-1}}p(\omega) = -m(N+m-2)p(\omega)$$
  \end{pthm}
  
  \begin{proof}\quad\par\vspace{3pt}
    Since
    $$\Delta = \frac{\partial^{2}}{\partial r^{2}}+ \frac{N-1}{r} \frac{\partial}{\partial r} + \frac{1}{r^{2}}\Delta_{S^{N-1}}\hspace{10pt}\text{and} \hspace{20pt} \Delta \left(r^{m}p(\omega)\right) = 0,$$
    we obtain
    $$  r^{N-2}\left\{m(N+m-2)p(\omega) + \Delta_{S^{N-1}}p(\omega)\right\} = 0.$$

  \end{proof}

  \begin{plemma}\label{spherep}

    Suppose $p \in\mathscr{H}^{m}(\mathbb{R}^{N})$ and $q \in\mathscr{H}^{m'}(\mathbb{R}^{N})$. Then, 

    $$ m \neq m' \hspace{10pt}\Rightarrow\hspace{10pt}\langle p,q\rangle_{S^{N-1}} = 0$$
    
  \end{plemma}
  \begin{proof}\quad\par\vspace{3pt}
    Since $ \Delta_{S^{N-1}}$ is symmetric operator about $\langle -,-\rangle_{S^{N-1}}$ and, with reference to Theorem \ref{eigen}, $p,q$ have different eigenvalues if  $(N,m,m')\neq (1,1,0),\,(1,0,1)$.
    Hence, these cases are proved. When $(N,m,m') = (1,1,0),\,(1,0,1)$, because $\mathscr{H}^{0}(\mathbb{R}^{1})=\mathbb{R}1$ and $\mathscr{H}^{1}(\mathbb{R}^{1})=\mathbb{R}x$, we may chcek the claim directly.

  \end{proof}

  \begin{theorem}\label{sphexpa}
    The following map
    $$ \bigoplus_{m=0}^{\infty} \mathscr{H}^{m}(\mathbb{R}^{N})\hookrightarrow L^2(S^{N-1}) \hspace{20pt} p\mapsto p\left|_{S^{N-1}}\right.$$
    is inclusion with dense image.

    \vspace{5pt}
    If $m \neq m'$ and $p \in \mathscr{H}^{m}(\mathbb{R}^{N})$, $q \in \mathscr{H}^{m'}(\mathbb{R}^{N})$, then the image of $p,q$ are orthogonal with each other.
  \end{theorem}

  \begin{proof}\quad\par\vspace{3pt}
  
    The orthogonality follows from Lemma \ref{spherep}. \par
    Also by Lemma \ref{spherep}, the proof of injectivity reduces to the injectivity of the map $\mathscr{H}^{m}(\mathbb{R}^{N})\rightarrow L^2(S^{N-1})$. 
    Since $\mathscr{H}^{m}(\mathbb{R}^{N})$ consists of homogeneous polynimials, we are able to construct inverse map from the image easily. This imply the injectivity. \par 
    \vspace{5pt}
    Now, the proof of denseness is remained.
    By Lemma \ref{polrep}, we need to prove the denseness of the image of
    $$ \mathscr{P}(\mathbb{R}^{N})\longrightarrow L^2(S^{N-1}) \hspace{20pt} p\mapsto p\left|_{S^{N-1}}\right. .$$
    This follows from the Stone-Weierstrass theorem.

  \end{proof}

%

\subsection{Gegenbauer polynomials}
We review some facts about Gegenbauer polynomials. 
The notation $C^{\nu}_{m}(t)$ is the same as \cite{MR2401813} and \cite{MR2956043}.

  \begin{definition}\label{defgeg}
    
    The \textbf{Gegenbauer polynomials} $C_{m}^{\nu}(t)$ are defined as the coefficient of the following formal expansion:
    $$\frac{1}{(1-2tx+x^2)^{\nu}}=\sum_{m=0}^{\infty} C^{\nu}_{m}(t)x^{m}, $$
    and the \textbf{normalized Gegenbauer polynomials} $\check{C}_{m}^{\nu}(t)$ is defined as the coefficient of the following formal expansion:
    $$\frac{1-x^{2}}{(1-2tx+x^2)^{\nu + 1}}=\sum_{m=0}^{\infty} \check{C}^{\nu}_{m}(t)x^{m}.$$

    These relate each other as follow:
    $$\check{C}^{\nu}_{m}(t)=\frac{m+\nu}{\nu}C^{\nu}_{m}(t)\,.$$
  \end{definition}

  \begin{proposition}\label{gegdiff}

    $$\frac{\partial}{\partial t} \check{C}^{\nu}_{m}(t) = 2(\nu +1) \check{C}^{\nu + 1}_{m-1}(t)$$
  
  \end{proposition}
  \begin{proof}\quad\par\vspace{3pt}
    Differntiating the formal expansion in Definition \ref{defgeg} by $t$, we obtain the result.
  \end{proof}

  \begin{theorem}[{\cite[Lemma 4.9]{MR2956043}}]\label{gegest}
    If $\nu \in \mathbb{R}$,
    \begin{align}
      \underset{-1\leqq t \leqq 1}{\mathrm{sup}}\left| \frac{1}{\nu}C_{m}^{\nu}(t)\right| & \leqq {}^\exists B(\nu)m^{2\nu-1}
    \end{align}
    
  \end{theorem}
  \begin{proof}\quad\par\vspace{3pt}
    Substituting $t=\cos\theta$, we obtain
    \begin{align}
      \MoveEqLeft (1-2tx+x^2)^{-\nu} = (1-xe^{i\theta})^{-\nu}(1-xe^{-i\theta})^{-\nu} \\
      & = \left(\sum_{k=0}^{\infty}\frac{(\nu)_{k}}{k!}x^{k}e^{ik\theta}\right)\left(\sum_{k=0}^{\infty}\frac{(\nu)_{l}}{l!}x^{l}e^{-il\theta}\right).
    \end{align} 
    Thus, 
    $$ C^{\nu}_{m}(\cos\theta)= \sum_{k=0}^{m}\frac{(\nu)_{k}(\nu)_{m-k}}{k!(m-k)!}\cos(m-2k)\theta. $$
    In particular,
    $$ \left|C^{\nu}_{m}(t)\right| \leqq C^{\nu}_{m}(1) $$
    (since the signature of coefficient $\frac{(\nu)_{k}(\nu)_{m-k}}{k!(m-k)!}$ do not depend on $k$).

    \vspace{30pt}
    By definition, 
    $$ C^{\nu}_{m}(1) = \frac{(2\nu)_{m}}{m!} = \frac{\Gamma(m+2\nu)}{m!\Gamma(2\nu)}.$$

    Applying stirling's asymptotic formula:
    $$ \underset{x\rightarrow\infty}{\lim} \frac{\Gamma(x)}{(2\pi)^{1/2}\,x^{x-\frac{1}{2}}e^{-x}} = 1,$$

    we obtain
    \begin{align}
      \MoveEqLeft \underset{m\rightarrow\infty}{\lim} m^{-(2\nu-1)}\frac{\Gamma(m+2\nu)}{m!\Gamma(2\nu)} = \frac{e^{-2\nu}}{\Gamma(2\nu)}\underset{m\rightarrow\infty}{\lim}\frac{m+1}{m}\left(1+\frac{2\nu}{m}\right)^{m+2\nu}= \frac{1}{\Gamma(2\nu)}.\\
    \end{align}
    Since $ \displaystyle\frac{\Gamma(m+2\nu)}{m!\Gamma(2\nu)}$ is a positive monotonic function about $m$, we obtain the result.
  \end{proof}

  \begin{corollary}\label{gegconv}
    The infinite sum
    $$\sum_{m=0}^{\infty} \left(\frac{\partial}{\partial t}\right)^{k}\check{C}^{\nu}_{m}(t)  \left(\frac{\partial}{\partial x}\right)^{l}x^{m} \hspace{20pt} (\,t\in[-1,1]\,,\,|x|<1\,)$$
    converges absolutely and uniformly on compacts and 
        $$  \left(\frac{\partial}{\partial t}\right)^{k} \left(\frac{\partial}{\partial x}\right)^{l}\frac{1-x^{2}}{(1-2tx+x^2)^{\nu + 1}} = \sum_{m=0}^{\infty} \left(\frac{\partial}{\partial t}\right)^{k}\check{C}^{\nu}_{m}(t)  \left(\frac{\partial}{\partial x}\right)^{l}x^{m}.$$
  
  \end{corollary}

  \begin{proof}\quad\par\vspace{3pt}
    By Proposition \ref{gegdiff} and Theorem \ref{gegest}, we obtain the results.
  \end{proof}

  \vspace{10pt}
  Next, we refer to the role of the Gegenbauer polynomials in the theory of spherical harmonics. $P_{m}(\omega,\mu)$ defined below coincide with $P_{k,m}(\omega,\mu)$ in \cite{MR2566988} when $k=0$.

  \begin{plemma}

    The function
    $$P(x,\mu):= \frac{1}{ \mathrm{vol} (S^{N-1}) } \frac{1-|x|^{2}}{|x-\mu|^{N}}\hspace{20pt}(x\in B^{N}, \vspace{5pt} \mu \in S^{N-1})$$
    is the Poisson kernel of a $N$-dimensional unit ball.

    That is, for any $f\in \mathrm{C}(S^{N-1})$,
    $$P_{f}(x):=\int_{S^{N-1}}P(x,\mu)f(\mu)d\mu$$
    is the unique harmonic function on $\mathrm{Int}(B^{N})$ such that $\underset{r\nearrow 1}{\lim} P_{f}(r\omega) = f(\omega)$.

  \end{plemma}

  \begin{proof}\quad\par\vspace{3pt}
    Firstly, $\Delta_{x} P(x,y) = 0$.

    Then, the integral 
    $$ I(x) = \int_{S^{n-1}} P(x,y)dy $$
    is $SO(n)$ - invariant harmonic function. 

    By the maximum principle of harmonic functions, it is a constant and $I(x) = I(0) = 1$.

    \vspace{10pt}
    For a continuous function $f\in C(S^{N-1})$, 
    
    \begin{align}
      \MoveEqLeft \left|f(\omega)-\int_{S^{N-1}}P(r\omega,\mu)f(\mu)d\mu\right| & \\
      & \leqq \int_{S^{N-1}}P(r\omega,\mu)\left|f(\omega)-f(\mu)\right|d\mu \\
      & \leqq \int_{\langle \omega, \mu \rangle \leqq 1-\delta}+  \int_{1-\delta \leqq \langle \omega, \mu \rangle \leqq 1} P(r\omega,\mu)\left|f(\omega)-f(\mu)\right|d\mu \\
      & \leqq \frac{2(1-r^{2}) \,\underset{ \omega\in S^{N}}{\mathrm{max}}|f(\omega)|}{\underset{\langle \omega, \mu \rangle \leqq 1-\delta}{\min}(1-2r\langle\omega,\mu\rangle + r^2)^{N/2}} + \underset{1-\delta \leqq \langle \omega, \mu \rangle \leqq 1}{\mathrm{max}}|f(\omega)-f(\mu)|
    \end{align}

    By taking a limits $r \rightarrow 0$ and $\delta \rightarrow 0 $ in this order, the assertion is proved.
  
  \end{proof}

  \begin{pcorollary}

    For $p(y)\in\mathscr{H}^{m}(\mathbb{R}^{N})$,
    $$p(x) = \int_{S^{N-1}}P(x,\mu)p(\mu)d \mu \hspace{20pt} (x\in B^{N})$$

  \end{pcorollary}

  The Poisson kernel has the following expansion,
  \begin{align}
    \MoveEqLeft P(r\omega,\mu) = \frac{1}{\mathrm{vol}(S^{N-1})}\frac{1-r^2}{\left(1-2r\langle\omega,\mu\rangle +r^2\right)^{N/2}}& \\
    &=  \frac{1}{\mathrm{vol}\left(S^{N-1}\right)}\sum_{m=0}^{\infty}  \check{C}^{(\nu)}_{m}(\langle\omega,\mu\rangle) r^{m} \hspace{20pt}\left(\text{Here, } \nu := \frac{N-2}{2}\right)
  \end{align}

  \vspace{10pt}
  Therefore, we define 
  $$P_{m}(\omega,\mu) :=  \check{C}^{(\nu)}_{m}(\langle\omega,\mu\rangle).$$
  Then the following corollaries hold.

  \begin{corollary}\label{corothogeg1}

    For $p(y)\in\mathscr{H}^{m}(\mathbb{R}^{N})$,
    $$\frac{1}{\mathrm{vol}\left(S^{N-1}\right)}\int_{S^{N-1}}P_{m'}(\omega,\mu)p(\mu)d\mu = 
    \begin{cases}
      p(\omega) & (m=m') \\
      0 & (m \neq m')\\
    \end{cases}$$

  \end{corollary}

  \begin{corollary}\label{corothogeg2}

    $$\frac{1}{\mathrm{vol}\left(S^{N-1}\right)}\int_{S^{N-1}}P_{m}(\omega,\mu)P_{m'}(\mu,\omega')d\mu = 
    \begin{cases}
      P_{m}(\omega,\omega ') & (m=m') \\
      0 & (m \neq m')\\
    \end{cases}$$

  \end{corollary}

  \begin{proof}\quad\par\vspace{3pt}

    From theorem \ref{sphexpa}
    we may write 
    $$P_{m}(\omega,\mu) = \sum_{i\in I_{m}} p_{i}(\omega)\overline{p_{i}(\mu)} $$

    Here $ \left\{p_{i}(\omega)\right\}_{i\in I_{m}}$ is a orthonormal basis of $\mathscr{H}^{m}(\mathbb{R}^{N})$ with respect to the Hermitian metric $\langle -,-\rangle_{S^{N-1}}$.
    From this formula, we obtain the asserstion.
  \end{proof}

  \vspace{10pt}
  \textbf{Remark)} These corollaries may be interpreted and derived from the viewpoint of zonal spherical harmonics.

\vspace{40pt}
\addcontentsline{toc}{section}{Reference}

\nocite{*}

\bibliography{paper1}

\newcommand{\etalchar}[1]{$^{#1}$}
\begin{thebibliography}{HKMM11b}

\bibitem[BD20]{MR4129081}
S.~{Ben Sa\"{\i}d} and L.~Deleaval.
\newblock Translation operator and maximal function for the
  ({$k$},1)-generalized {F}ourier transform.
\newblock {\em J. Funct. Anal.}, 279(8):108706, 32, 2020.

\bibitem[BJ98]{MR1631302}
A.~Braverman and A.~Joseph.
\newblock The minimal realization from deformation theory.
\newblock {\em J. Algebra}, 205(1):13--36, 1998.

\bibitem[BK94a]{MR1278630}
R.~Brylinski and B.~Kostant.
\newblock Minimal representations, geometric quantization, and unitarity.
\newblock {\em Proc. Nat. Acad. Sci. U.S.A.}, 91(13):6026--6029, 1994.

\bibitem[BK94b]{MR1267034}
R.~Brylinski and B.~Kostant.
\newblock Minimal representations of {$E_6$}, {$E_7$}, and {$E_8$} and the
  generalized {C}apelli identity.
\newblock {\em Proc. Nat. Acad. Sci. U.S.A.}, 91(7):2469--2472, 1994.

\bibitem[BK{\O}09]{MR2566988}
S.~{Ben Sa\"{\i}d}, T.~Kobayashi, and B.~{\O}rsted.
\newblock Generalized {F}ourier transforms {$\mathscr{F}_{k,a}$}.
\newblock {\em C. R. Math. Acad. Sci. Paris}, 347(19-20):1119--1124, 2009.

\bibitem[BK{\O}12]{MR2956043}
S.~{Ben Sa\"{\i}d}, T.~Kobayashi, and B.~{\O}rsted.
\newblock Laguerre semigroup and {D}unkl operators.
\newblock {\em Compos. Math.}, 148(4):1265--1336, 2012.

\bibitem[BZ91]{MR1108044}
B.~Binegar and R.~Zierau.
\newblock Unitarization of a singular representation of {${\mathrm{SO}}(p,q)$}.
\newblock {\em Comm. Math. Phys.}, 138(2):245--258, 1991.

\bibitem[CDL18]{MR3759078}
D.~Constales, H.~{De Bie}, and P.~Lian.
\newblock Explicit formulas for the {D}unkl dihedral kernel and the
  {$(\kappa,a)$}-generalized {F}ourier kernel.
\newblock {\em J. Math. Anal. Appl.}, 460(2):900--926, 2018.

\bibitem[DD18]{MR3767365}
L.~Deleaval and N.~Demni.
\newblock On a {N}eumann-type series for modified {B}essel functions of the
  first kind.
\newblock {\em Proc. Amer. Math. Soc.}, 146(5):2149--2161, 2018.

\bibitem[DLM25]{MR4832104}
H.~{De Bie}, P.~Lian, and F.~Maes.
\newblock Bounds for the kernel of the ({$\it \kappa $},{$a$})-generalized
  {F}ourier transform.
\newblock {\em J. Funct. Anal.}, 288(4):Paper No. 110755, 29, 2025.

\bibitem[{\relax DLMF}]{NIST:DLMF}
{\it NIST Digital Library of Mathematical Functions}.
\newblock \url{https://dlmf.nist.gov/}, Release 1.2.1 of 2024-06-15.
\newblock F.~W.~J. Olver, A.~B. {Olde Daalhuis}, D.~W. Lozier, B.~I. Schneider,
  R.~F. Boisvert, C.~W. Clark, B.~R. Miller, B.~V. Saunders, H.~S. Cohl, and
  M.~A. McClain, eds.

\bibitem[Dun89]{MR951883}
C.~F. Dunkl.
\newblock Differential-difference operators associated to reflection groups.
\newblock {\em Trans. Amer. Math. Soc.}, 311(1):167--183, 1989.

\bibitem[Dun92]{MR1199124}
C.~F. Dunkl.
\newblock Hankel transforms associated to finite reflection groups.
\newblock In {\em Hypergeometric functions on domains of positivity, {J}ack
  polynomials, and applications ({T}ampa, {FL}, 1991)}, volume 138 of {\em
  Contemp. Math.}, pages 123--138. Amer. Math. Soc., Providence, RI, 1992.

\bibitem[DX13]{MR3060033}
F.~Dai and Y.~Xu.
\newblock {\em Approximation theory and harmonic analysis on spheres and
  balls}.
\newblock Springer Monographs in Mathematics. Springer, New York, 2013.

\bibitem[Fol89]{MR0983366}
G.~B. Folland.
\newblock {\em Harmonic analysis in phase space}, volume 122 of {\em Annals of
  Mathematics Studies}.
\newblock Princeton University Press, Princeton, NJ, 1989.

\bibitem[Fun15]{Funaki}
T.~Funaki.
\newblock {\em Stochastic differential equations (in Japanese)}.
\newblock Iwanami Shoten, Publishers.,Tokyo, 2015.

\bibitem[GKK{\etalchar{+}}25]{MR4867008}
W.~M. Goldman, K.~Kannaka, T.~Kubo, T.~Okuda, Y.~Oshima, M.~Pevzner, A.~Sasaki,
  and H.~Sekiguchi.
\newblock The mathematical work of {T}oshiyuki {K}obayashi.
\newblock In {\em Symmetry in geometry and analysis. {V}ol. 1. {F}estschrift in
  honor of {T}oshiyuki {K}obayashi}, volume 357 of {\em Progr. Math.}, pages
  1--102. Birkh\"auser/Springer, Singapore, [2025] \copyright 2025.

\bibitem[GR07]{MR2360010}
I.~S. Gradshteyn and I.~M. Ryzhik.
\newblock {\em Table of integrals, series, and products}.
\newblock Elsevier/Academic Press, Amsterdam, seventh edition, 2007.
\newblock Translated from the Russian, Translation edited and with a preface by
  Alan Jeffrey and Daniel Zwillinger, With one CD-ROM (Windows, Macintosh and
  UNIX).

\bibitem[GS05]{MR2123125}
W.~T. Gan and G.~Savin.
\newblock On minimal representations definitions and properties.
\newblock {\em Represent. Theory}, 9:46--93, 2005.

\bibitem[GW94]{MR1327538}
B.~H. Gross and N.~R. Wallach.
\newblock A distinguished family of unitary representations for the exceptional
  groups of real rank {$=4$}.
\newblock In {\em Lie theory and geometry}, volume 123 of {\em Progr. Math.},
  pages 289--304. Birkh\"{a}user Boston, Boston, MA, 1994.

\bibitem[HKMM11a]{MR2860690}
J.~Hilgert, T.~Kobayashi, G.~Mano, and J.~M\"{o}llers.
\newblock Orthogonal polynomials associated to a certain fourth order
  differential equation.
\newblock {\em Ramanujan J.}, 26(3):295--310, 2011.

\bibitem[HKMM11b]{MR2837716}
J.~Hilgert, T.~Kobayashi, G.~Mano, and J.~M\"{o}llers.
\newblock Special functions associated with a certain fourth-order differential
  equation.
\newblock {\em Ramanujan J.}, 26(1):1--34, 2011.

\bibitem[How88]{MR0974332}
R.~Howe.
\newblock The oscillator semigroup.
\newblock In {\em The mathematical heritage of {H}ermann {W}eyl ({D}urham,
  {NC}, 1987)}, volume~48 of {\em Proc. Sympos. Pure Math.}, pages 61--132.
  Amer. Math. Soc., Providence, RI, 1988.

\bibitem[KM07a]{MR2317306}
T.~Kobayashi and G.~Mano.
\newblock Integral formula of the unitary inversion operator for the minimal
  representation of {${\mathrm{O}}(p,q)$}.
\newblock {\em Proc. Japan Acad. Ser. A Math. Sci.}, 83(3):27--31, 2007.

\bibitem[KM07b]{MR2401813}
T.~Kobayashi and G.~Mano.
\newblock The inversion formula and holomorphic extension of the minimal
  representation of the conformal group.
\newblock In {\em Harmonic analysis, group representations, automorphic forms
  and invariant theory}, volume~12 of {\em Lect. Notes Ser. Inst. Math. Sci.
  Natl. Univ. Singap.}, pages 151--208. World Sci. Publ., Hackensack, NJ, 2007.

\bibitem[KM11]{MR2858535}
T.~Kobayashi and G.~Mano.
\newblock The {S}chr\"{o}dinger model for the minimal representation of the
  indefinite orthogonal group {${\mathrm{O}}(p,q)$}.
\newblock {\em Mem. Amer. Math. Soc.}, 213(1000):vi+132, 2011.

\bibitem[K{\O}98]{MR1649917}
T.~Kobayashi and B.~{\O}rsted.
\newblock Conformal geometry and branching laws for unitary representations
  attached to minimal nilpotent orbits.
\newblock {\em C. R. Acad. Sci. Paris S\'{e}r. I Math.}, 326(8):925--930, 1998.

\bibitem[K{\O}03a]{MR2020550}
T.~Kobayashi and B.~{\O}rsted.
\newblock Analysis on the minimal representation of {$\mathrm{O}(p,q)$}. {I}.
  {R}ealization via conformal geometry.
\newblock {\em Adv. Math.}, 180(2):486--512, 2003.

\bibitem[K{\O}03b]{MR2020551}
T.~Kobayashi and B.~{\O}rsted.
\newblock Analysis on the minimal representation of {$\mathrm{O}(p,q)$}. {II}.
  {B}ranching laws.
\newblock {\em Adv. Math.}, 180(2):513--550, 2003.

\bibitem[K{\O}03c]{MR2020552}
T.~Kobayashi and B.~{\O}rsted.
\newblock Analysis on the minimal representation of {$\mathrm{O}(p,q)$}. {III}.
  {U}ltrahyperbolic equations on {${\Bbb R}^{p-1,q-1}$}.
\newblock {\em Adv. Math.}, 180(2):551--595, 2003.

\bibitem[Kob98]{MR1637667}
T.~Kobayashi.
\newblock Discrete decomposability of the restriction of
  {$A_{\mathfrak{q}}(\lambda)$} with respect to reductive subgroups. {II}.
  {M}icro-local analysis and asymptotic {$K$}-support.
\newblock {\em Ann. of Math. (2)}, 147(3):709--729, 1998.

\bibitem[Kob03]{MR1982432}
T.~Kobayashi.
\newblock Conformal geometry and global solutions to the {Y}amabe equations on
  classical pseudo-{R}iemannian manifolds.
\newblock In {\em Proceedings of the 22nd {W}inter {S}chool ``{G}eometry and
  {P}hysics'' ({S}rn\'{\i}, 2002)}, number~71, pages 15--40, 2003.

\bibitem[Kob11]{MR2849643}
T.~Kobayashi.
\newblock Algebraic analysis of minimal representations.
\newblock {\em Publ. Res. Inst. Math. Sci.}, 47(2):585--611, 2011.

\bibitem[Kos90]{MR1103588}
B.~Kostant.
\newblock The vanishing of scalar curvature and the minimal representation of
  {${\mathrm{SO}}(4,4)$}.
\newblock In {\em Operator algebras, unitary representations, enveloping
  algebras, and invariant theory ({P}aris, 1989)}, volume~92 of {\em Progr.
  Math.}, pages 85--124. Birkh\"{a}user Boston, Boston, MA, 1990.

\bibitem[Kot15]{Kotani}
S.~Kotani.
\newblock {\em Measures and Probability (in Japanese)}.
\newblock Iwanami Shoten, Publishers.,Tokyo, 2015.

\bibitem[R\"98]{MR1620515}
M.~R\"osler.
\newblock Generalized {H}ermite polynomials and the heat equation for {D}unkl
  operators.
\newblock {\em Comm. Math. Phys.}, 192(3):519--542, 1998.

\bibitem[Sab96]{MR1387517}
H.~Sabourin.
\newblock Une repr\'{e}sentation unipotente associ\'{e}e \`a l'orbite minimale:
  le cas de {${\mathrm{SO}}(4,3)$}.
\newblock {\em J. Funct. Anal.}, 137(2):394--465, 1996.

\bibitem[Seg61]{MR0128839}
I.~E. Segal.
\newblock Foundations of the theory of dyamical systems of in- finitely many
  degrees of freedom. {II}.
\newblock {\em Canadian J. Math.}, 13:1--18, 1961.

\bibitem[Sha62]{MR0137504}
D.~Shale.
\newblock Linear symmetries of free boson fields.
\newblock {\em Trans. Amer. Math. Soc.}, 103:149--167, 1962.

\bibitem[Tam19]{MR3968907}
H.~Tamori.
\newblock Classification of minimal representations of real simple {L}ie
  groups.
\newblock {\em Math. Z.}, 292(1-2):387--402, 2019.

\bibitem[Tor97]{MR1484858}
P.~Torasso.
\newblock M\'{e}thode des orbites de {K}irillov-{D}uflo et repr\'{e}sentations
  minimales des groupes simples sur un corps local de caract\'{e}ristique
  nulle.
\newblock {\em Duke Math. J.}, 90(2):261--377, 1997.

\bibitem[TT25]{MR4871318}
K.~Taira and H.~Tamori.
\newblock Strichartz estimates for the {$(k,a)$}-generalized {L}aguerre
  operators.
\newblock {\em SIGMA Symmetry Integrability Geom. Methods Appl.}, 21:Paper No.
  014, 37, 2025.

\bibitem[VH51]{MR0057260}
L.~Van~Hove.
\newblock Sur certaines repr\'{e}sentations unitaires d'un groupe infini de
  transformations.
\newblock {\em Acad. Roy. Belg. Cl. Sci. M\'{e}m. Coll. in 8$^\circ$},
  26(6):102, 1951.

\bibitem[Vog94]{MR1253210}
D.~A. Vogan, Jr.
\newblock The unitary dual of {$G_2$}.
\newblock {\em Invent. Math.}, 116(1-3):677--791, 1994.

\bibitem[Wei64]{MR0165033}
A.~Weil.
\newblock Sur certains groupes d'op\'{e}rateurs unitaires.
\newblock {\em Acta Math.}, 111:143--211, 1964.

\bibitem[WW96]{MR1424469}
E.~T. Whittaker and G.~N. Watson.
\newblock {\em A course of modern analysis}.
\newblock Cambridge Mathematical Library. Cambridge University Press,
  Cambridge, 1996.

\end{thebibliography}
\bibliographystyle{alpha}

\end{document}